\documentclass[12pt,letterpaper]{article}	

\usepackage{linguex} 
\usepackage{times} 
\usepackage{lipsum} 
\usepackage{ragged2e}
\usepackage{amsmath}
\usepackage{amsthm}
\usepackage{amsfonts}
\usepackage[utf8]{inputenc}
\usepackage{graphicx}
\usepackage{algorithm}
\usepackage{algpseudocode}
\usepackage{float}
\usepackage{subfloat}
\usepackage{subfig}
\usepackage{tikz}
\usepackage{bm}
\usepackage{pgfplots}
\usetikzlibrary{ decorations.markings}
\usepackage{cases}
\usepackage{empheq,etoolbox}

\definecolor{lightBlue}{HTML}{000080}
\definecolor{lightGreen}{HTML}{228B22}
\definecolor{lightRed}{HTML}{F08080} 
\definecolor{darkRed}{HTML}{960018}
\definecolor{lightlightBlue}{HTML}{708090} %
\definecolor{darkBlue}{HTML}{191970} %
\definecolor{lightGray}{gray}{0.9}
\definecolor{mediumGray}{gray}{0.4}
\definecolor{darkGray}{gray}{0.25}

\RaggedRight

\usepackage{scrextend} 
\deffootnote[.5em]{0em}{1em}{\textsuperscript{\thefootnotemark}\,}

\newcounter{savefootnote} 
\newcounter{symfootnote}
\newcommand{\symfootnote}[1]{%
	\setcounter{savefootnote}{\value{footnote}}%
	\setcounter{footnote}{\value{symfootnote}}%
	\ifnum\value{footnote}>8\setcounter{footnote}{0}\fi%
	\let\oldthefootnote=\thefootnote%
	\renewcommand{\thefootnote}{\fnsymbol{footnote}}%
	\footnote{#1}%
	\let\thefootnote=\oldthefootnote%
	\setcounter{symfootnote}{\value{footnote}}%
	\setcounter{footnote}{\value{savefootnote}}%
}
\newtheorem{theorem}{Theorem}

\newtheorem{property}{Property}

\newtheorem{definition}{Definition}
\newtheorem{remark}{Remark}

\usepackage[labelsep=period]{caption}
\numberwithin{equation}{section}

\usepackage[margin=1.0in]{geometry}
\usepackage[compact]{titlesec}
\titleformat{\section}[runin]{\normalfont\bfseries}{\thesection.}{.5em}{}[.]
\titleformat{\subsection}[runin]{\normalfont\scshape}{\thesubsection.}{.5em}{}[.]
\titleformat{\subsubsection}[runin]{\normalfont\scshape}{\thesubsubsection.}{.5em}{}[.]
\usepackage[colorlinks,allcolors={black},urlcolor={blue}]{hyperref} 

\usepackage[
backend=bibtex,
style=ieee,
citestyle=numeric,
sorting=nty,
maxbibnames=99,
mincitenames=1,
maxcitenames=6
]{biblatex}
\renewenvironment{abstract}{%
	\noindent\begin{minipage}{1\textwidth}
		\setlength{\leftskip}{0.4in}
		\setlength{\rightskip}{0.4in}
		\textbf{Abstract.}}
	{\end{minipage}}

\newenvironment{keywords}{%
	\vspace{.5em}
	\noindent\begin{minipage}{1\textwidth}
		\setlength{\leftskip}{0.4in}
		\setlength{\rightskip}{0.4in}
		\textbf{Keywords.}}
	{\end{minipage}}

\begin{document}	
	\setlength{\Extopsep}{6pt}
	\setlength{\Exlabelsep}{9pt} 
	
	\begin{center}
		\normalfont\bfseries
		Numerical Solution of linear drift-diffusion and pure drift equations on one-dimensional graphs
		\vskip .5em
		\normalfont
		{Beatrice Crippa$^\star$, Anna Scotti$^\star$, Andrea Villa$^\dagger$} \\
		\vskip .5em
		\footnotesize{$^\star$ MOX-Laboratory for Modeling and Scientific Computing, Department of Mathematics, Politecnico di Milano, 20133 Milan, Italy \\
		$^\dagger$ Ricerca Sul Sistema Energetico (RSE), 20134 Milano, Italy}
	\end{center}
	
	\justifying
	
	\begin{abstract}
		We propose numerical schemes for the approximate solution of problems defined on the edges of a one-dimensional graph. In particular, we consider linear transport and a drift-diffusion equations, and discretize them by extending Finite Volume schemes with upwind flux to domains presenting bifurcation nodes with an arbitrary number of incoming and outgoing edges, and implicit time discretization. We show that the discrete problems admit positive unique solutions, and we test the methods on the intricate geometry of an electrical treeing.
	\end{abstract}
	
	
	\section{Introduction}
	\label{section:Introduction}
    The present work concerns the numerical solution of time-dependent drift and drift-diffusion problems, with linear advection terms, on one-dimensional graphs. \\
	This work is motivated by the simulation of defects embedded in insulation systems of power network. In fact, the reliability of the electric system is highly influenced by internal propagation of defects in insulating components, a phenomenon known as treeing~\cite{bahadoorsingh2007role, buccella2023computational}. Electrical treeing structures loosely resemble real trees and are characterized by highly branched geometries with very long and thin channels, see~\cite{schurch2014imaging, schurch20193d}.  These channels evolve in time, due to internal Partial Discharges~\cite{raymond2015partial} inside non-thermalized plasmas, where electrons and ions are free to move, causing avalanche effects of multiplications of the charged particles~\cite{pankaj2017cold}.
    There are plenty of semi-empirical models in the literature ~\cite{nyanteh2011overview, jow1996stochastic} but still, to the best of the authors knowledge, there is lack of models based on first principles. Some examples can be found~\cite{callender2019plasma, villa2022towards}, but they are limited to either very simplified configurations or very small treeing structures corresponding to the first stages of propagation. Unfortunately, fully developed treeing geometries are so complex that even the creation of a coherent mesh is almost impossible to achieve. For these reasons, the use of simplified, one-dimensional appproximations is a more viable approach. Our aim is to adapt the techniques developed in~\cite{VILLA2017687, villa2017implicit, villa2015stability, villa2017simulation} for the simulation of treeing structures, and the first step in this direction is to develop drift and drift-diffusion solvers for 1D branched structures, describing the evolution of the concentration of charged particles in the electrical treeing. In particular, the volume concentration of charged species inside the channels follows a drift-diffusion equation, while the surface charge concentration on the interface between the defect and the external plasma a pure transport equation. In both cases the advection speed is directly dependent on the electric field acting on the domain. A crucial property that these solvers must satisfy is positivity of the solution, necessary in view of the future coupling of the described problem with the chemical effects~\cite{buccella2023computational} due to Partial Discharges and electron avalanche effect.

    The one-dimensional domain can be described as a network, made of interconnected nodes, with differential equations defined on its edges and proper conditions at intersections. A complete introduction on network structures can be found in~\cite{estrada2012structure} and~\cite{newman2018networks}. \\
    Further applications of problems on networks, with possibly non-linear drift, can be found in traffic flow on road networks~\cite{bressan2014flows, piccoli2006traffic}, gas flow in pipe networks~\cite{gottlich2005network, herty2010new}, supply chain models~\cite{herty2007existence}, water networks~\cite{hild2012real} and neuron models~\cite{wybo2015sparse}. For these problems, ad-hoc solvers were proposed~\cite{bretti2005numerical, goatin2020comparative, gyrya2019explicit, leugering2017domain}, based on Godunov schemes~\cite{godunov1959finite}, discrete velocities kinetic methods~\cite{aregba2004kinetic}, Finite Differences~\cite{strikwerda2004finite} and Domain Decomposition~\cite{lagnese2004domain}. Moreover, an extension of Finite Element Methods to differential equations on quantum graphs~\cite{berkolaiko2013introduction} is discussed in~\cite{arioli2017finite}.
    
    In this paper we focus on the solution of problems with linear advection term, without any additional assumption on the advection velocities on each branch of the domain. We introduce a numerical method that combines an extension of the Finite Volume scheme with upwind flux for the space discretization and implicit Euler for the time discretization, providing unconditional stability. We prove existence and uniqueness of the numerical solutions and the postivity of the scheme, meaning that, starting from a non-negative initial condition, at every time step the solution is non-negative. Moreover, we discuss the consistency of the numerical fluxes at graph nodes, also in presence of bifurcations with an arbitrary number of incoming and outgoing edges.
    
    The rest of this manuscript is organized as follows: in Section~\ref{section:ModelProblems} we present the structure of the domain and define the model problems on it; the numerical methods are introduced in Section~\ref{section:NumericalMethods} and the existence, uniqueness and positivity of their solution, together with consistency of the numerical fluxes, is investigated in Section~\ref{section:NumericalAnalysis}. Finally, we present in Section~\ref{section:Results} some tests of these methods on increasingly complex geometries, from a straight line to graphs with only one intersection node to the application to the geometry of a typical electrical tree.
	
	\section{Model problems}
	\label{section:ModelProblems}
	We will focus on time dependent problems on a finite time interval $[0,T], \ T\in\mathbb{R}$, and on a spatial domain represented by a one-dimensional, connected, simple, directed graph $\Lambda = \left(\mathcal{E},\mathcal{N}\right)$, consisting of a set of edges $\mathcal{E}=\{e_k\}_{k=1}^{n_e}$ and a set of nodes $\mathcal{N}=\{v_k\}_{k=0}^{n_v}$. Each edge is parametrized by its arc length parameter $l_k$ and normalized with a map $\pi_k \ : \ [0,1]\to e_k$. We suppose that the parametrization is opposite to the direction of the advective flux, in agreement with previous literature, and we say that an edge $e_k$ is \textit{incoming} to a node $v$ if $\pi_k(0) = v$, \textit{outgoing} if $\pi_k(1) = v$, namely if the flux on $e_k$ is directed towards $v$ or away from it, respectively. We denote by $\mathcal{E}_v^+$ and $\mathcal{E}_v^-$ the sets of incoming and outgoing edges of a node $v\in\mathcal{N}$:

    \begin{equation*}
		\mathcal{E}_v^+ := \{e_{j}\in\mathcal{E} \ : \ \pi_j(1) = v\},
	\end{equation*}
	
   \begin{equation*}
	 	\mathcal{E}_v^- := \{e_{j}\in\mathcal{E} \ : \ \pi_j(0) = v\}.
    \end{equation*}

	\noindent an by $\mathcal{E}_v = \mathcal{E}_v^+ \cup \mathcal{E}_v^-$ the set of all the edges intersecting at $v$. In an analogous way we can define the sets $\mathcal{E}^+_k$ and $\mathcal{E}^-_k$, $k=1,\dots,n_e$, of incoming edges into the inflow endpoint, and outgoing edges from the outflow node of an edge $e_k$, respectively:
    
    \begin{equation*}
		\mathcal{E}_k^+ := \{e_{j}\in\mathcal{E} \ : \ \pi_j(0) = \pi_k(1)\},
	\end{equation*}
	
    \begin{equation*}
        \mathcal{E}_k^- := \{e_{j}\in\mathcal{E}_k \ : \ \pi_j(1) = \pi_k(0)\}.
    \end{equation*}

	On such graph we can define the spatial derivative $\frac{\partial \cdot}{\partial s}$ on each edge as the derivative on each edge with respect to the parametrization.
	
	We consider a pure transport problem:\\
	
	Find $u\ :\ \Lambda\times[0,T]\longrightarrow \mathbb{R}$ such that
	
	\begin{equation}
		\dfrac{\partial u}{\partial t} + \dfrac{\partial }{\partial s}\left(c u\right) = f, \qquad \text{on \ } \Lambda\times(0,T];
		\label{eq:transport}
	\end{equation}
 
    \noindent with transport velocity $c: \ \Lambda\to[0,+\infty)$ and source term given by a function $f:\ \Lambda\times[0,T]\to\mathbb{R}$, and a drift-diffusion equation: \\
	
	Find $u\ :\ \Lambda\times[0,T]\longrightarrow \mathbb{R}$ such that

	\begin{equation}
		\dfrac{\partial u}{\partial t} - \dfrac{\partial}{\partial s} \left(\nu \frac{\partial u}{\partial s}\right) + \dfrac{\partial }{\partial s}\left(c u\right) = f, \qquad \text{on \ } \Lambda\times(0,T].
		\label{eq:diffusion_transport}
	\end{equation}

    \noindent with coefficients $c: \ \Lambda\to[0,+\infty)$ and $\nu: \ \Lambda\to[0,+\infty)$ and source term $f:\ \Lambda\times[0,T]\to\mathbb{R}$.
	
	\noindent The set of nodes $\mathcal{N}$ can be partitioned into two non-overlapping sets: $\mathcal{N} = \mathcal{N}_i \cup \mathcal{N}_b, \ \mathcal{N}_i\cap\mathcal{N}_b=\emptyset$, where we call $\mathcal{N}_i$ the set of \textit{internal} nodes, i.e. nodes shared by at least two segments, and $\mathcal{N}_b$ the set of \textit{boundary} nodes, i.e. connected to only one edge of the graph.\\
	On all the nodes in $\mathcal{N}_b$ we must impose boundary conditions for problems~\eqref{eq:transport} and~\eqref{eq:diffusion_transport}. For each problem we denote by $\mathcal{N}_s$ the set of \textit{Dirichlet} nodes, and by $\mathcal{N}_e$ the set of \textit{end nodes}, where we impose Neumann conditions:
	
	\begin{equation}
		\begin{cases}
			u = \bar{u}, &\text{on\ every\ }v\in\mathcal{N}_s, \\
			\dfrac{du}{ds}=0, &\text{on\ every\ }v\in\mathcal{N}_e.
		\end{cases}
		\label{eq:bc}
	\end{equation}
	
	\noindent We assume that the nodes where we impose Dirichlet conditions are \textit{sources}, meaning they have no incoming edges, and those with Neumann (or Robin) conditions are \textit{sinks}, with no outgoing edges. In the advection problem~\eqref{eq:transport}, we can also have sink nodes where no boundary condition is imposed, if they are outflow nodes, where outflow conditions are not necessary.\\
	Moreover, in both cases, we allow one node to be shared among more than two edges, and we call such a node a \textit{bifurcation}. We denote the set of bifurcations by $\mathcal{N}_{\text{bif}} \subset \mathcal{N}_i$. In each internal node there must be at least one incoming and one outgoing edge, in order to avoid mass accumulation there.

			
			

    We denote by $f_k$ the restriction $f\big|_{e_k}$ of a function $f:\ \Lambda\to\mathbb{R}$ to the edge $e_k,\ k=1,\dots,n_e$. In particular, we can write the velocity and concentration on each edge as $c_k = c\big|_{e_k}$ and $u_k = u\big|_{e_k}$, respectively. Then, the conservation of fluxes at a node $v\in\mathcal{N}$ is:

    \begin{equation*}
		\sum_{e_k\in\mathcal{E}_v^+} u_k c_k = \sum_{e_j\in\mathcal{E}_v^-} u_j c_j,
	\end{equation*}

	\noindent and the complete formulation of the transport problem~\eqref{eq:transport} on the network reads as follows:
	
	\begin{subnumcases}{}
		\dfrac{\partial u_k}{\partial t} + \dfrac{\partial}{\partial s}\left(c_k u_k\right) = f_k,\ &$ \forall k\in\{1,\dots,n_e\}, \ \text{in}\ [0,1]\times(0,T] $, 		\label{eq:graph:transport:begin}\\
      \sum_{e_k\in\mathcal{E}_v^+}c_k u_k = \sum_{e_j\in\mathcal{E}_v^-}c_j u_j, \ & $ \forall v\in \mathcal{N}_i,\ \text{in}\ (0,T] $, \\
		u_k(\pi_k(1),t) = \bar{u}_k(t), & $ \forall k\in\{1,\dots,n_e\} \ : \  \pi_k(1)\in\mathcal{N}_s,\ \forall t \in (0,T] $, \\
		u_k(x,0) = u^0(x), &$ \forall x\in[0,1], \ \forall k\in\{1,\dots,n_e\} $,
		\label{eq:graph:transport:end}
	\end{subnumcases}

	\noindent while the complete formulation of the drift-diffusion~\eqref{eq:diffusion_transport} problem on the network is:
	
	\begin{subnumcases}{}
			\dfrac{\partial u_k}{\partial t} - \dfrac{\partial}{\partial s}\left(\nu_k \frac{\partial u_k}{\partial s}\right) + \dfrac{\partial}{\partial s}\left(c_k u_k\right) = f_k, &$ \forall k\in\{1,\dots,n_e\}, \ \text{in}\ [0,1]\times(0,T] $, 
			\label{eq:graph:diffusion_transport:begin} \\
		      \sum_{e_k\in\mathcal{E}_v^+}c_k u_k = \sum_{e_j\in\mathcal{E}_v^-}c_j u_j, \ & $ \forall v\in \mathcal{N}_i,\ \text{in}\ (0,T] $, \\
			u_k(\pi_k(1),t) = \bar{u}_k(t), & $ \forall k\in\{1,\dots,n_e\} \ : \  \pi_k(1)\in\mathcal{N}_s,\ \forall t \in (0,T] $, \\
			\dfrac{\partial u_k}{\partial s}(\pi_k(0),t) = 0, & $ \forall k\in\{1,\dots,n_e\} \ : \  \pi_k(0)\in\mathcal{N}_e,\ \forall t \in (0,T] $,\\
			u_k(x,0) = u^0(x), &$ \forall x\in[0,1], \ \forall k\in\{1,\dots,n_e\} $.
		\label{eq:graph:diffusion_transport:end}
	\end{subnumcases}

    \noindent All the hypotheses introduced above are necessary for stating the well-posedness of these continuous problems on a graph. Indeed, according to \textcite{vonbelow1988classical}, the drift-diffusion problem~\eqref{eq:graph:diffusion_transport:begin}-~\eqref{eq:graph:diffusion_transport:end} admits a unique solution if the coefficients $\nu$ and $c$ belong to the space $\mathcal{C}^{\alpha,\alpha/2}(\Lambda\times[0,T]) $ of Holder continuous functions for some $\alpha$, while the existence of a unique solution of the transport problem~\eqref{eq:graph:transport:begin}-~\eqref{eq:graph:transport:end} is ensured by \textcite{banasiak2014asymptotic}.
    
    \begin{remark}
        In~\cite{banasiak2014asymptotic} the possibility of dealing with homogeneous Neumann conditions on an open graph is discussed: in this case, we would have injection of some quantity, whose concentration is represented by the unknown $u$, from the source nodes, and no outflow from the sinks, so we need a workaround in order to make the problem well-posed and avoid mass accumulation on the vertices. The authors propose to extend the graph by adding one extra vertex for each sink, connected to it by two edges going in opposite directions, as to create a cycle (see Figure~\ref{figure:extended}), and they prove that these additional vertices do not influence the behavior of the solution on the original graph, the flow on each appended subgraph is asymototically periodic in time and the edges incoming to the sinks will be eventually depleted as $t\to\infty$, or in a finite time if they are in the acyclic part of the graph. However, as we will see in Section~\ref{section:NumericalAnalysis}, this solution is not required in the discrete setting.
    \end{remark}

	\begin{figure}
		\centering
		\begin{tikzpicture}
            [decoration={markings, mark= at position 0.5 with {\arrow{stealth}}}]
			\filldraw (0,0) circle (2pt) node [xshift = .3cm, yshift = .3cm]{$v_2$};
			\filldraw (0,2) circle (2pt) node [xshift = .3cm]{$v_1$};
			\filldraw (-2,0) circle (2pt) node[xshift = .3cm, yshift = .3cm]{$v_3$};
			\filldraw (2,0) circle (2pt) node [xshift = -.3cm, yshift = .3cm]{$v_4$};
			\filldraw (-4,0) circle (2pt) node[xshift = -.3cm]{$v_5$};
			\filldraw (4,0) circle (2pt) node [xshift = .3cm]{$v_6$};
   
			\draw [postaction={decorate}, thick] (0,0) -- (-2,0) node [xshift = .4cm, yshift = 1cm]{};
			\draw [postaction={decorate}, thick] (0,0) -- (2,0) node [xshift = -.4cm, yshift = 1cm]{};
			\draw [postaction={decorate}, thick] (0,2) -- (0,0) node [xshift = -.4cm, yshift = 1cm]{};
   
			\draw [postaction={decorate}, thick, dashed] plot[smooth] coordinates {(4,0) (3,.5) (2,0)};
			\draw [postaction={decorate}, thick, dashed] plot[smooth] coordinates {(2,0) (3,-.3) (4,0)};
			\draw [postaction={decorate}, thick, dashed] plot[smooth] coordinates {(-4,0) (-3,.5) (-2,0)};
			\draw [postaction={decorate}, thick, dashed] plot[smooth] coordinates {(-2,0) (-3,-.3) (-4,0)};
		\end{tikzpicture}
		\caption{Simple one-dimensional domain with three edges and one bifurcation node $v_2$. If we impose homogeneous Neumann conditions on the set of sinks $\mathcal{N}_e = \{v_3,\ v_4\}$ the graph should be extended by also including the nodes $v_5$ and $v_6$ and the additional dashed edeges, according to~\cite{banasiak2014asymptotic}.}
		\label{figure:extended}
	\end{figure}
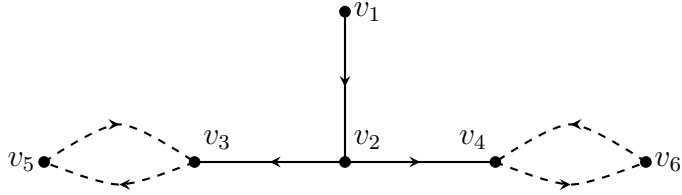
 
	\section{Numerical methods}
	\label{section:NumericalMethods}
	
	We start by discretizing the two model problems \eqref{eq:transport} and \eqref{eq:diffusion_transport} in time by introducing a set of equispaced instants $\{t_l\ :\ l=0,\dots,N\}$  on $[0,T]$, such that $t_0 = 0$, $t_{N} = T$ and $t_{l+1} = t_{l} + \Delta t,\ \forall l\in\{0,\dots,N-1\}$. We denote by $u^l$ the solution $u$ at time $t_l,\ l=0,\dots,N$ and apply an Implicit Euler discretization scheme on both problems:
	
	\begin{subequations}
		\begin{align}
			\dfrac{u^{l+1} - u^{l}}{\Delta t} + \dfrac{d}{ds}\left(u^{l+1} \ c\right) &= f^{l+1}, \qquad \text{on\ } \Lambda,\ \forall l\in\{0,\dots,N-1\},
			\label{eq:semid:transport}\\
			\dfrac{u^{l+1} - u^{l}}{\Delta t} - \dfrac{d}{ds}\left(\nu \frac{d u^{l+1}}{ds}\right) + \dfrac{d}{ds}\left(u^{l+1}\ c\right) &= f^{l+1}, \qquad \text{on\ } \Lambda,\ \forall l\in\{0,\dots,N-1\}.
			 \label{eq:semid:diffusion_transport}
		\end{align}
	\end{subequations}

	\noindent The choice of an implicit solver is due to the final aim of defining methods for the simulation of thermal plasma without restrictions on the time step length, due to stability issues \cite{villa2017implicit, villa2021uncoupled}.
	
	
	
	\subsection{Space discretization}
	\label{section:NM:SpaceDiscretization}
	
	We complete the discretization of equations \eqref{eq:semid:transport} and \eqref{eq:semid:diffusion_transport} by applying a Finite Volume (FV) Method on the 1D domain $\Lambda$, represented by a set of connected segments with possible branches.\\
%
%
%
%
%
%
	Based on the concepts discussed in the previous section, we will introduce a numerical scheme for the approximate solution of the aforementioned PDEs. We start by the discretization of a linear transport problem in Section~\ref{section:NM:Transport}, followed by the diffusion equations in Section~\ref{section:NM:Diffusion}, which introduces the solver for drift-diffusion problems, finally presented in Section~\ref{section:NM:DiffusionTransport}.

	\subsubsection{Transport equation}
	\label{section:NM:Transport}
		
	Consider time-dependent transport equations in the form of~\eqref{eq:transport}, semi-discretized in time as~\eqref{eq:semid:transport}.\\
	For simplicity, as a partition for $\Lambda$ we consider the set of segments representing its edges $\{e_k\}_{k=1}^{n_e}$ and we complete the discretization of the equation by integrating~\eqref{eq:semid:transport} on each segment:
	
	\begin{equation}
		\int_{e_k}\dfrac{u^{l+1}-u^{l}}{\Delta t} = -\int_{e_k} \dfrac{d}{ds}\left(u^{l+1}c\right), \quad \forall k\in\{1,\dots,n_e\},\ \forall l\in\{0,\dots,N-1\}.
		\label{eq:fv:transport:1}
	\end{equation}
	
	\noindent Notice that a further partition of the edges can be trivially introduced by considering more internal nodes on each edge.
	
	On each segment we approximate the value of $u$ as a constant, given by its integral mean. From now on, with a little abuse of notation, we will denote by $u_k^{l}$ the constant approximation of the solution on each segment $e_k$ at time $t_l$: $u^l\big|_{e_k} \approx u_k^l$, $\forall k\in\{1,\dots,n_e\},\ l\in\{0,\dots,N\}$. We choose to approximate the right hand side integral as the sum of numerical fluxes at the endpoints of each segment, with an upwind scheme. Integrating the first derivative of the flux $F(u)=uc$ on the edge $e_k$, we would obtain the difference between its value on the endnodes $v_k^2=\pi_k(0)$ and $v_k^1=\pi_k(1)$. Since the numerical solution may have jump discontinuities on nodes separating two neighboring edges, we have to define a numerical flux $\tilde{F}$ to set its value there.
    According to the upwind method, we set the flux on the fist endnode $v_k^1$ equal to the upstream flux $\tilde{F}(u_v) = c_k u_v$, with concentration $u$ evaluated in the incoming neighbors of $e_k$ and on the second $v_k^2$ equal to the upstream flux $\tilde{F}(u_k) = c_k u_k$ on $e_k$, thus obtaining the following approximation:
	
	\begin{equation}
		\int_{e_k} \dfrac{d}{ds}\left(u^{l+1}c\right) \approx c_k
		\left(u_k^{l+1} - u_v^{l+1}\right),
		\quad \forall k\in\{1,\dots,n_e\},\ \forall l\in\{0,\dots,N\},
		\label{eq:fv:transport:2}
	\end{equation}

	\noindent where $u_v^{l+1}$ indicates the concentration $u$ at time $t_{l+1}$ from the incoming neighbors of $e_k$ through the vertex $v=\pi_k(1)$. We will define it as a function of the concentrations on the neighboring edges of $e_k$, based on the flux conservation condition at the junction node $v$.
 
	We can substitute~\eqref{eq:fv:transport:2} into~\eqref{eq:fv:transport:1} and $u^l$ with its constant approximation $u_k^l$ on each segment to obtain the following expression:
	
	\begin{equation}
		\dfrac{|e_k|}{\Delta t} u_k^{l+1}+ c_k
		\left(u_k^{l+1} - u_v^{l+1}\right) = \dfrac{|e_k|}{\Delta t} u_k^{l}, \quad \forall k\in\{1,\dots,n_e\},\quad \forall l\in\{0,\dots,N\}.
		\label{eq:transport:numerical_flux}
	\end{equation}
	
	

	\noindent On the node $v=\pi_k(1)$ we can impose the conservation of fluxes and determine the portion of mass transported from the incoming edges $e_i$ to the outgoing edges $e_k$. The numerical fluxes $\tilde{F}_v^+$ and $\tilde{F}_v^-$, entering and leaving the node are given by the sum of fluxes on its incoming and outgoing edges, respectively:
	
	\begin{equation}
		\begin{aligned}
			\tilde{F}_v^+ & = \sum_{e_j\in\mathcal{E}_v^+} \tilde{F}_j = \sum_{e_j\in\mathcal{E}_v^+} c_j u_j, \\
			\tilde{F}_v^- & = \sum_{e_j\in\mathcal{E}_v^-} \tilde{F}_j = \sum_{e_j\in\mathcal{E}_v^-} c_j u_v.
		\end{aligned}
		\label{eq:transport:fluxes}
	\end{equation}


    
	\noindent If we balance the fluxes~\eqref{eq:transport:fluxes}, we obtain:
	
	\begin{equation}
        u_v^{l+1} = \sum_{e_j\in\mathcal{E}_v^+}\left( \dfrac{c_j}{\sum_{e_k\in\mathcal{E}_v^-} c_k} u_j^{l+1} \right).
		\label{eq:flux_balance:discrete}
	\end{equation}

	

	
	
	\noindent Then, the value of $u$ at a node $v$ is given by a weighted sum of its value on the incoming edges, where the weights are given by
	
	\begin{equation*}
		w_{j} = \dfrac{c_j}{\sum_{e_l\in\mathcal{E}_v^-} c_l},\quad \forall j\ :\ e_j \in \mathcal{E}_v^+, \ \forall v\in\mathcal{N}\setminus\mathcal{N}_e.
	\end{equation*}

	\noindent If we substitute the expression of $u_v^{l+1}$ obtained by this relation in equation~\eqref{eq:transport:numerical_flux}, we end up with the following numerical scheme:

	\begin{equation}
		\left(\dfrac{|e_k|}{\Delta t} + c_k\right) u_k^{l+1} - c_k \sum_{e_i\in\mathcal{E}_k^+} w_{i} u_i^{l+1} = \dfrac{|e_k|}{\Delta t} u_k^l, \quad \forall k\in\{1,\dots,n_e\},\ \forall l\in\{0,\dots,N\},
		\label{eq:discr:transport}
	\end{equation}
	
	\noindent which can be written in matrix form as:
	
	\begin{equation}
		M_\text{tr}\mathbf{u}^{l+1}=\mathbf{g}^{l}, \quad \forall l\in\{0,\dots,N-1\},
		\label{eq:alg:transport}
	\end{equation}
	\noindent where $\mathbf{g}^{l}=\left[\frac{|e_1|}{\Delta t}u_1^{l},\dots,\frac{|e_{n_e}|}{\Delta t}u_{n_e}^{l}\right]^T\in\mathbb{R}^{n_e}$ is a vector depending on the unknowns at the previous time $t_l$, and $M\in\mathbb{R}^{n_e \times n_e}$ is a square matrix with entries given by:

    \begin{equation}
		(M_\text{tr})_{ki} =
		\begin{cases}
			\dfrac{|e_k|}{\Delta t} + c_k, & \text{if}\ i=k,\ k=1,\dots,n_e, \\
			-c_k w_{i}, &
			\text{if}\ i\neq k,\ i,k=1,\dots,n_e,\ e_i\in\mathcal{E}_k^+, \\
			0, & \text{otherwise}.
		\end{cases}
		\label{eq:transport:matrix:A}
	\end{equation}



	\subsubsection{Diffusion equation}
	\label{section:NM:Diffusion}

	Let us introduce a numerical scheme for the diffusion term, which will then be combined with the method presented in the previous section.\\
	Consider the following steady equation:
	\begin{equation}
		- \dfrac{d}{ds}\left(\nu \dfrac{du}{ds}\right) = 0,\quad \text{on}\ \Lambda,
		\label{eq:diffusion}
	\end{equation}

	\noindent completed by the Dirichlet and Neumann boundary conditions~\ref{eq:bc} introduced in Section~\eqref{section:ModelProblems}.
 
	We consider first, for clarity of exposition, the simple case where $\Lambda$ is made of three segments, one having vertices numbered as $v_0$ and $v_1$ and the other two branching from vertex $v_1$, as in Figure~\ref{figure:branch}. We denote by $A$ the vertex with index 0, $I$ the vertex with index 1 and by $B$ and $C$ the endpoints of the two edges of the branch. Then, the segments composing $\Lambda$ are $e_1 = AI$, $e_2=IB$ and $e_3=IC$.
	
	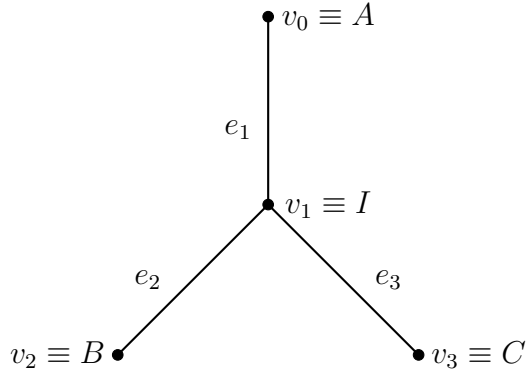
\begin{figure}
		\centering
		\begin{tikzpicture}
			\filldraw (0,0) circle (2pt) node [xshift = .8cm] {$v_1\equiv I$};
			\filldraw (0,2.5) circle (2pt) node [xshift = .8cm] {$v_0\equiv A$};
			\filldraw (-2,-2) circle (2pt) node[xshift = -.8cm] {$v_2\equiv B$};
			\filldraw (2,-2) circle (2pt) node [xshift = .8cm] {$v_3\equiv C$};
			\draw [thick] (0,0) -- (-2,-2) node [xshift = .4cm, yshift = 1cm] {$e_2$};
			\draw [thick] (0,0) -- (2,-2) node [xshift = -.4cm, yshift = 1cm] {$e_3$};
			\draw [thick] (0,2.5) -- (0,0) node [xshift = -.4cm, yshift = 1cm] {$e_1$};
		\end{tikzpicture}
		\caption{Simple one-dimensional domain with three edges and one bifurcation node $v_1$. The set of edges is $\mathcal{E}=\{e_1=AI,e_2=IB,e_3=IC\}$. We set Dirichlet boundary conditions on $A$ and Neumann boundary conditions on the set of end nodes $\mathcal{N}_e=\{B,C\}$. The only bifurcation node is $I$.}
		\label{figure:branch}
	\end{figure}

	On this domain, problem~\eqref{eq:diffusion}-\eqref{eq:bc} reads as follows:
	
	\begin{equation}\nonumber
		\begin{dcases}
			- \dfrac{d}{ds}\left(\nu \dfrac{du}{ds}\right) &= 0,\quad \text{on}\ \Lambda, \\
			u(A) &= \bar{u}, \\
			\dfrac{du}{ds}(B) =\dfrac{du}{ds}(C) &= 0.
		\end{dcases}
	\end{equation}
	
	\noindent If we now integrate the first equation of this system on each segment applying the fundamental theorem of calculus, we obtain:
	
	\begin{equation}\nonumber
		\begin{aligned}
			-\int_{e_k}\dfrac{d}{ds}\left(\nu \dfrac{du}{ds}\right)&=\int_{e_k} 0, \\
			\nu(\pi_k(0))\dfrac{du}{ds}(\pi_k(0))-\nu(\pi_k(1))\dfrac{du}{ds}(\pi_k(1)) &= 0.
		\end{aligned}
	\end{equation}
	
	\noindent Let us denote by $\nu_{k,j}$ the value of the diffusion coefficient defined on edge $e_j$, evaluated at node $v=\pi_k(j), \ j =0,1$. Furthermore, we can approximate 
	the first derivative of $u$ at a node by Finite Differences, similarly to the Two-Point Flux Approximation (TPFA) approach \cite{eymard2000finite}.
	In particular, considering two adjacent and collinear segments, as in Figure~\ref{figure:segment}, we would approximate its value on the node between $e_k$ and $e_{k+1}$ as
	
	\begin{equation*}
		\dfrac{du}{ds} \approx T_{k,k+1} \left(u_{k+1} - u_k\right), \quad \forall k=1,\dots,n_e-1,
		\label{eq:diffusion:flux_approx}
	\end{equation*}

	\noindent where $T_{k,k+1}$ denotes the inverse of the distance between the centers of $e_k$ and $e_{k+1}$, and in this case it is simply
	
	\begin{equation}
		T_{k,k+1} = \left(\dfrac{|e_{k}| + |e_{k+1}|}{2}\right)^{-1}.
		\label{eq:distanceT}
	\end{equation}

    \noindent If the two segments are not collinear, the approximation of the first derivative remains the same and the distance $T_{k,j}^{-1}$ between the centers of consecutive elements is considered as distance on the graph, as in equation~\eqref{eq:distanceT}. 
	If a node belongs to one edge only, then it must be either a source node $v\in\mathcal{N}_s$ or an end node $v\in\mathcal{N}_e$. In the first case, 
	
	\begin{equation}
		\dfrac{du}{ds}(v) \approx T_{k,1} \left(u_k - \bar{u}_k\right), \quad \forall v\in\mathcal{N}_s,
    \label{eq:diff:der:2}
	\end{equation}

	\noindent where we denote by $T_{k,1} = \dfrac{2}{|e_k|}$ the inverse of distance from the center of the segment $e_k$ to the source node. 
	Otherwise, if it is an end node, we impose homogeneous Neumann boundary conditions and we can directly substitute
	
	\begin{equation}
		\dfrac{du}{ds}(v) = 0, \quad \forall v\in\mathcal{N}_e.
    \label{eq:diff:der:3}
	\end{equation}

	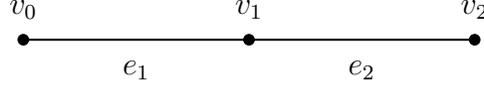
\begin{figure}
		\centering
		\begin{tikzpicture}
			\filldraw (0,0) circle (2pt) node [yshift = .4cm] {$v_0$};
			\filldraw (3,0) circle (2pt) node [yshift = .4cm] {$v_1$};
			\filldraw (6,0) circle (2pt) node [yshift = .4cm] {$v_2$};
			\draw [thick] (0,0) -- (3,0) node [xshift = -1.5cm, yshift = -.4cm] {$e_1$};
			\draw [thick] (3,0) -- (6,0) node [xshift = -1.5cm, yshift = -.4cm] {$e_2$};
		\end{tikzpicture}
		\caption{Segment discretized by contiguous elements without bifurcations. The set of edges is $\mathcal{E}=\{e_1 = AB, \ e_2=BC\}$ and there are no bifurcations. We impose Dirichlet boundary condition on the source node $v_0$ and homogeneous Neumann on the sink $v_2$.}
		\label{figure:segment}
	\end{figure}

    \noindent Finally, in our case, a node may be shared by more than two edges. Hence, we will introduce the following generalization.\\
    Consider the node $I$ in the graph of Figure~\ref{figure:branch}, connected to three edges. Here we separately consider the three contributions and then balance the flux at the intersection $v_1 = I$. If we assume that the orientations of the edges $e_1, \ e_2 \ \text{and}\ e_3$ are $\overrightarrow{AI},\ \overrightarrow{IB}\ \text{and}\ \overrightarrow{IC}$, respectively, then $v_1 = \pi_1(0) = \pi_2(1) = \pi_3(1)$. In this case the only incoming edge is $e_1$, while outgoing ones are $e_2$ and $e_3$. If we assume that there is no mass accumulation at the nodes, then the fluxes are balanced, meaning that the sum of incoming and outgoing fluxes there is null:
	
	\begin{equation}
	   \nu_1(v_1)\dfrac{du_1}{ds}(v_1) = \nu_2(v_1)\dfrac{du_2}{ds}(v_1) + \nu_3(v_1) \dfrac{du_3}{ds}(v_1).
		\label{eq:flux_balance}
	\end{equation}
	
	\noindent We can consider the bifurcation nodes as boundary points of subgraphs of $\Lambda$ not containing bifurcations. On these additional boundary points we impose Dirichlet boundary conditions:
	
	\begin{equation}
		u_k(v) = u(v), \quad \forall v\in \mathcal{N}_{\text{bif}}.
		\label{eq:bc:bif}
	\end{equation}

	\noindent Clearly we do not know the value of the function $u$ at these nodes, but we can compute it making use of the flux balance condition~\eqref{eq:flux_balance}. This means that we need to consider an additional set of discrete unknowns $\{u_{n_e+j}\}_{j=1}^{n_{\text{bif}}}$, where $n_{\text{bif}}=|\mathcal{N}_{\text{bif}}|$ is the number of bifurcation nodes in $\Lambda$, such that $u_{n_e+j} = u(v_j), \ \forall j=1,\dots,n_{\text{bif}}$.\\
    %
    The approximation of derivatives on $v\in\mathcal{N}_\text{bif}$ is given by:

    \begin{equation}
        \dfrac{d u_k}{ds}(v_j) \approx
        \begin{dcases}
            T_{k,n_e+j}(u_k - u_{n_e+j}), & \text{if\ } v_j = \pi_k(1), \\
            T_{k,n_e+j}(u_{n_e+j} - u_k), & \text{if\ } v_j = \pi_k(0).
        \end{dcases}
    \label{eq:diff:der:4}
    \end{equation}
    
	\noindent Here we have denoted by $T_{k,n_e+j}=\dfrac{2}{|e_k|}$ the inverse of the distance between the center of element $e_k$ and the bifurcation node $v_j$ along the graph.
	

	Combining~\eqref{eq:diff:der:2},~\eqref{eq:diff:der:3} and~\eqref{eq:diff:der:4}, we obtain a linear system of the form
	
	\begin{equation}
		A_{\text{diff}}\mathbf{u} = \mathbf{0},
		\label{eq:diffusion:matrix}
	\end{equation}

	\noindent where $ A_{\text{diff}} $ is an $(n_e+n_{\text{bif}})\times(n_e+n_{\text{bif}})$ matrix with the following block structure:
	
	\begin{equation}
		A_{\text{diff}} =
		\begin{bmatrix}
			B & C \\
			C^T & D
		\end{bmatrix},
		\label{eq:diffusion:matrix:B}
	\end{equation}
	
	\noindent with $B\in \mathbb{R}^{n_e\times n_e}$, $C\in\mathbb{R}^{n_e \times n_{\text{bif}}}$ and $D\in\mathbb{R}^{n_{\text{bif}} \times n_{\text{bif}}}$ diagonal matrix. The entries of these matrices are:
	
	\begin{equation}
		\begin{array}{lll}
			C_{k,j} &=
			\begin{dcases}
				-\dfrac{\nu_{k,j}}{|e_k|} T_{k,j}, &\text{if\ } \pi_k(1)=v_j\in\mathcal{N}_\text{bif} \text{\ or\ } \pi_k(0)=v_j\in\mathcal{N}_\text{bif}, \\
				0, &\text{otherwise},
			\end{dcases}\quad
			&
			\begin{split}
				&\forall k=1,\dots,n_e, \\
				&\forall j=1,\dots,n_\text{bif};
			\end{split}\\
	
			D_{k,j} &=
			\begin{dcases}
				- \sum_{i=1}^{n_e} C_{i,j}, &\text{if\ } k=j, \\
				0, &\text{otherwise},
			\end{dcases}\quad
			& \forall k,j=1,\dots,n_\text{bif}; \\
	
			B_{k,j} &=
			\begin{dcases}
				-\dfrac{\nu_{k,j}}{|e_k|} T_{k,j}, &\text{if\ } e_j\in \mathcal{E}_k, v = \ e_j\cap e_k \notin \mathcal{N}_\text{bif}, \\
				0, &\text{otherwise}
			\end{dcases}\quad
			&
			\begin{split}
				\forall k,j=1,\dots,n_e, \\
				j\neq k;
			\end{split}\\
			
			B_{k,k} &=
			\begin{dcases}
				- \sum_{j=1, \ j\neq k}^{n_e} B_{k,j} - \sum_{j=n_e+1}^{n_\text{bif}} C_{k,j}, & \text{if\ } \pi_k(1)\notin\mathcal{N}_s,\\
				- \sum_{j=1, \ j\neq k}^{n_e} B_{k,j} - \sum_{j=n_e+1}^{n_\text{bif}} C_{k,j} + \nu_{k,1}T_{k,1}, &\text{otherwise}
			\end{dcases}\quad
			& \forall k=1,\dots,n_e.
		\end{array}
	\label{diffusion:B}
	\end{equation}

	\noindent The vector $\mathbf{u}\in\mathbb{R}^{n_e+n_\text{bif}}$ introduced in equation~\eqref{eq:diffusion:matrix} contains the unknowns, ordered as the $n_e$ values of $u$ at the edges $e_k,\ k=1,\dots,n_e,\ \mathbf{\hat{u}}=[u_1,\dots,u_{n_e}]^T$, and at the bifurcation points $\mathbf{u}_B=[u_{n_e+1},\dots,u_{n_e+n_\text{bif}}]^T$:
	
	\begin{equation*}
		\mathbf{u} =
		\begin{bmatrix}
			\mathbf{\hat{u}} \\
			\mathbf{u}_B
		\end{bmatrix}
		.
	\end{equation*}
	
	
	
	
 
	\subsubsection{drift-diffusion equation}
	\label{section:NM:DiffusionTransport}

    By combining equations~\eqref{eq:alg:transport} and~\eqref{eq:diffusion:matrix} we can discretize the drift-diffusion equation~\eqref{eq:diffusion_transport}, containing a system of Ordinary Differential Equations in the unknowns $\mathbf{\hat{u}}$ and $\mathbf{u}_B$. Note that in the discretization of the transport term the only considered unknowns are $\mathbf{\hat{u}}$, representing the approximate values of $u$ on the edges. Therefore, we need to augment the corresponding matrix in order to be able to sum it to the diffusion part:
	

	
	
	
	
	
	\begin{equation}
		M =
		\begin{bmatrix}
			I+M_\text{tr} & 0 \\
			0 & 0
		\end{bmatrix}
		+
		\Delta t A_\text{diff},
		\label{eq:diffusion_transport:matrix:M}
	\end{equation}

    \noindent where $M_\text{tr}\in\mathbb{R}^{n_e\times n_e}$ and $A_\text{diff}\in\mathbb{R}^{(n_e+n_{\text{bif}})\times(n_e+n_{\text{bif}})}$ are defined in equations~\eqref{eq:transport:matrix:A} and~\eqref{eq:diffusion:matrix:B}, respectively, while the vector of unknowns and the right-hand side are defined as in Section~\ref{section:NM:Diffusion}, equation~\eqref{eq:diffusion:matrix}.

    \section{Properties of the numerical scheme}
	\label{section:NumericalAnalysis}
	In this section we focus on proving some properties of the matrices defined in equations~\eqref{eq:transport:matrix:A},~\eqref{eq:diffusion:matrix:B} and~\eqref{eq:diffusion_transport:matrix:M}, providing existence and uniqueness of the solutions and positivity of the associated linear systems. We will show these results based on some properties of Z-matrices and M-matrices \cite{PLEMMONS1977175}, recalled for readers' convenience:
	
	\begin{definition}
		A matrix $A\in\mathbb{R}^{N\times N}$ is a \textbf{Z-matrix} if $A\in Z^{N\times N}$, with \\
		\begin{equation}\nonumber
			Z^{N\times N}:=\{M\in\mathbb{R}^{N\times N}\ :\ M_{ij}\leq 0 \ \forall i,j \ \text{such\ that}\ i\neq j,\ 1\leq i,j \leq N\}.
		\end{equation}
	\end{definition}
		
	\begin{definition}
		A matrix $M\in\mathbb{R}^{N\times N}$ is an \textbf{M-matrix} if $\ \exists s\in\mathbb{R},\ B\in\mathbb{R}^{N\times N}$ such that $M=sI-B$, with $B_{ij}\geq 0$, $1\leq i,j\leq N$, and $s\geq \rho(B)$, where $\rho(B)$ denotes the spectral radius of $B$.
	\end{definition}
	
	\subsection{Transport equation}
	\label{section:NA:Transport}
	Let us start by analyzing the discrete transport equation introduced in Section~\ref{section:NM:Transport}. Observe that the matrix $M_\text{tr}$, defined in equation~\eqref{eq:transport:matrix:A}, has all non-negative diagonal entries and non-positive extra-diagonal entries. As a consequnce,

    \begin{property}
        The matrix $M_\text{tr}$ defined in equation~\eqref{eq:transport:matrix:A} is a Z-matrix.
        \label{property:transport:Zmatrix}
    \end{property}
 
    \noindent Moreover, in the following we will prove that $M_\text{tr}$ has all positive column sums. These properties allow us to prove that it is invertible and that, in equation~\eqref{eq:alg:transport}, $\mathbf{g}\geq 0 \Rightarrow \mathbf{u}\geq 0$. Consequently the system $M_\text{tr} \mathbf{u} = \mathbf{g}$ admits a unique solution, which is positive if the source term is. In fact,

    \begin{property}
        The matrix $M_\text{tr}$ defined in equation~\eqref{eq:transport:matrix:A} has positive column sums.
        \label{property:A:colsum}
    \end{property}

    \begin{proof}
        Fixed any column $i\in\{1,\dots,n_e\}$ of $M_\text{tr}$, the sum of its elements is given by

        \begin{equation*}
            \sum_{k=1}^{n_e} (M_\text{tr})_{ki} = (M_\text{tr})_{ii} +
            \sum_{
            \scriptsize
            \begin{split}
                k=1\\
                k\neq i
            \end{split}
            }^{n_e} (M_\text{tr})_{ki} = 
            \dfrac{|e_i|}{\Delta t} + c_i - \sum_{k\in\mathcal{E}_v^-} c_k w_i =
            \dfrac{|e_i|}{\Delta t} + c_i - c_i \dfrac{\sum_{k\in\mathcal{E}_v^-} c_k}{\sum_{k\in\mathcal{E}_v^-} c_k} = \dfrac{|e_i|}{\Delta t} >0,
        \end{equation*}

        \noindent where $v=\pi_k(0)$.
        
    \end{proof}
	
	\noindent This result allows us to prove the following theorem:
	
	\begin{theorem}
		Let $M_\text{tr}$ be the $n_e\times n_e$ matrix defined by~\eqref{eq:transport:matrix:A}. Then, $M_\text{tr}$ is invertible and the system $M_\text{tr}\mathbf{u}=\mathbf{g}$ is positive.
		\label{theorem:transport}
	\end{theorem}

	\begin{proof}
		Since $M_\text{tr}$ is a Z-matrix (Property~\ref{property:A:colsum}), so is its transpose $M_\text{tr}^T$ is. According to Theorem 1 of \textcite{PLEMMONS1977175}, this property is necessary and sufficient to state that $M_\text{tr}^T$ is a nonsingular M-matrix, which is also equivalent to: $M_\text{tr}^T$ is inverse-positive, i.e. $\exists (M_\text{tr}^T)^{-1}$, $((M_\text{tr}^T)^{-1})_{ij}\geq 0$ and $M_\text{tr}^{-1}\neq 0$.\\
        By definition of inverse, we have that $I = M_\text{tr}^T (M_\text{tr}^T)^{-1} $, where $I$ denotes the $n_e\times n_e$ identity matrix. If we transpose both sides of the equations, we get 

        \begin{equation*}
            I = \left(M_\text{tr}^T (M_\text{tr}^T)^{-1}\right)^T = \left((M_\text{tr}^T)^{-1}\right)^T M_\text{tr}.
        \end{equation*}

        \noindent This is only possible if $M_\text{tr}$ is invertible and its inverse is $M_\text{tr}^{-1} = \left((M_\text{tr}^T\right)^{-1})^T$. Then, also $M_\text{tr}$ is inverse-positive, i.e. $(M_\text{tr}^{-1})_{ij}\geq 0,\quad \forall i,j\in\{1,\dots,n_e\}$, and, as a consequence,
		
		\begin{equation*}
			\mathbf{u}_k=\sum_{i=1}^{n_e} (M_\text{tr}^{-1})_{ki}\mathbf{f}_i\geq 0, \quad \forall k\in\{1,\dots,n_e\},
		\end{equation*}
		
		\noindent because they are linear combinations of non-negative quantities.
	\end{proof}

    \noindent Finally, we can prove consistency of the proposed numerical flux on the graph nodes.
	
	\begin{theorem}
		The flux approximation~\eqref{eq:fv:transport:2} is consistent at every internal graph node.
		\label{theorem:transport:consistency}
	\end{theorem}

	\begin{proof}
		Denote the exact flux at an internal node $v\in\mathcal{N}_i$ by $F_v = (cu)(s_v)$, where $s_v\in\Lambda$ is the position of the node on the graph and consider the corresponding incoming and outgoing numerical fluxes $\tilde{F}^+_v$ and $\tilde{F}^-_v$  defined in equation~\eqref{eq:transport:fluxes}.\\		
		We introduce the evaluation of the numerical fluxes with respect to the exact solution:
		
		\begin{equation*}
			\tilde{F}^{\star,+}_v = \sum_{ e_i\in\mathcal{E}_v^+} (cu)(s_i),
		\end{equation*}
	
		\begin{equation*}
			\tilde{F}^{\star,-}_v = \sum_{ e_i\in\mathcal{E}_v^-} c(s_i)u(s_v),
		\end{equation*}	
				
		\noindent and the approximate flux at the bifurcation node should satisfy $F^\star_v = F^{\star,+}_v = F^{\star,-}_v$, by balance of fluxes.\\
		Then, we obtain the following expression for the exact solution at the node:
		
		\begin{equation*}
			u(s_v) = \dfrac{\sum_{ e_i\in\mathcal{E}_v^+} c(s_i)u(s_i)}{\sum_{ e_i\in\mathcal{E}_v^-} c(s_i)},
		\end{equation*}
	
		\noindent which can be substituted in the definition of the exact flux, to obtain:
		
		\begin{equation*}
			F_v = c(s_v)\dfrac{\sum_{ e_i\in\mathcal{E}_v^+} c(s_i)u(s_i)}{\sum_{ e_i\in\mathcal{E}_v^-} c(s_i)} = \dfrac{c(s_v)}{\sum_{ e_i\in\mathcal{E}_v^-} c(s_i)} F^\star_v.
		\end{equation*}
	
		\noindent Finally, the truncation error can be computed as:
		
		\begin{equation*}
			\tau_v = F_v - F^\star_v = \left( \dfrac{c(s_v)}{\sum_{ e_i\in\mathcal{E}_v^-} c(s_i)} -1  \right) F^\star_v = \left( c(s_v) - \sum_{ e_i\in\mathcal{E}_v^-} c(s_i) \right) u(s_i).
		\end{equation*}

        \noindent We have no hypotheses on the continuity of the speed $c$ at nodes, so it may present jumps at those points. We need to introduce the following compatibility condition, setting the value of the velocity at nodes as:
		
		\begin{equation}
			c(s_v) = \sum_{ e_i\in\mathcal{E}_v^-}c(s_i).
			\label{eq:th:tr:c}
		\end{equation}
  
        \noindent where $s_i$ indicates the position of the midpoint of the edge $e_i$ on $\Lambda$. 
        Notice that, if the node $v$ has only one outgoing edge, the velocity at the node is set equal to the outgoing velocity.\\
        By this definition of the speed at nodes, we get $\tau_v=0$ and we can conclude that the the method is strongly consistent at the graph nodes.
				
	\end{proof}
	
	\begin{remark}
        The proposed numerical scheme does not take into account the outflow conditions on sinks, but it can be adapted to enforce bondary conditions, such as Robin or homogeneous Neumann conditions, by imposing them on the edges connected to the sinks. The approximation error on these edges should decrease with the dimension of the space discretization and consistency still hold.
	\end{remark}
	
	\subsection{Diffusion equation}
	\label{section:NA:Diffusion}
	Let us now analyze the diffusion matrix $A_\text{diff}$, defined by equation~\eqref{eq:diffusion:matrix:B}. This is a block $\left(n_e + n_\text{bif}\right)\times\left(n_e + n_\text{bif}\right)$ matrix, with the following properties:
	
	\begin{property}
		The matrix $A_\text{diff}$ defined in equation~\eqref{eq:diffusion:matrix:B} is a Z-matrix;
		\label{property:B:Zmatrix}
	\end{property}
	
	\begin{property}
		The matrix $A_\text{diff}$ defined in equation~\eqref{eq:diffusion:matrix:B} is symmetric.
		\label{property:B:symmetric}
	\end{property}

	\noindent These two properties follow from the definition of the matrix. Moreover, we can prove the following result:

	\begin{property}
		The row sums of the elements of the matrix $A_\text{diff}$ defined in equation~\eqref{eq:diffusion:matrix:B} are always zero but for a number $n=|\mathcal{N}_s|>0$ of rows, and the same holds by columns.	
		\label{property:B:row_sum}
	\end{property}
	\begin{proof}
		Let us start by the last $n_\text{bif}$ rows, consisting of the blocks
		
		\begin{equation*}
			\begin{bmatrix}
				C^T & D
			\end{bmatrix}
			,
		\end{equation*}

		\noindent where the only non-null element of $D$ on every row is on its diagonal and is equal, by definition (equation~\eqref{diffusion:B}), to the opposite of the column sum of $C$. Then,
		
		\begin{equation*}
            \begin{split}
			     \sum_{k=1}^{n_e+n_\text{bif}} (A_\text{diff})_{jk} = \sum_{k=1}^{n_e} C^T_{jk} + \sum_{k=1}^{n_\text{bif}} D_{jk} = \sum_{k=1}^{n_e} C_{kj} + D_{kk} = \sum_{k=1}^{n_e} C_{kj} -\sum_{k=1}^{n_e} & C_{kj} = 0, \\
            &j=n_e+1,\dots,n_e+n_\text{bif}.
            \end{split}
		\end{equation*}
		
		\noindent Among the first $n_e$ rows we find the non-null row sums, corresponding to the elements on whose endpoints Dirichlet boundary conditions are imposed. Indeed, by definition (equation~\eqref{diffusion:B}), the diagonal elements of $A_{\text{diff}}$ are defined as the opposite of the sum of the non-diagonal entries of each row, plus an additive term in case of Dirichlet conditions imposed on an endnode: 
		
		\begin{multline*}
			\sum_{k=1}^{n_e+n_\text{bif}} (A_\text{diff})_{jk} = (A_\text{diff})_{kk} +
            \sum_{
                \scriptsize
            \begin{split}
                k=1\\
                k\neq j
            \end{split}
            }^{n_e+n_\text{bif}} (A_\text{diff})_{jk} = B_{kk} +
            \sum_{
                \scriptsize
            \begin{split}
                k=1\\
                k\neq j
            \end{split}
            }^{n_e+n_\text{bif}} (A_\text{diff})_{jk} = \\
			=
			\begin{dcases}
				- \sum_{
                \scriptsize
                \begin{split}
                    k=1\\
                    k\neq j
                \end{split}
                }^{n_e+n_\text{bif}} (A_\text{diff})_{jk} +
                \sum_{
                \scriptsize
                \begin{split}
                    k=1\\
                    k\neq j
                \end{split}
                }^{n_e+n_\text{bif}} (A_\text{diff})_{jk} = 0, & \text{if\ } \pi_k(1)\notin\mathcal{N}_s, \\
				- \sum_{
                \scriptsize
                \begin{split}
                    k=1\\
                    k\neq j
                \end{split}
                }^{n_e+n_\text{bif}} (A_\text{diff})_{jk} + \nu_{k,1} T_{k,1} +
                \sum_{
                \scriptsize
                \begin{split}
                    k=1\\
                    k\neq j
                \end{split}
                }^{n_e+n_\text{bif}} (A_\text{diff})_{jk} = \nu_{k,1} T_{k,1} > 0, & \text{otherwise}.
			\end{dcases}
		\end{multline*}
	
		\noindent Since the only rows with positive sum are the ones corresponding to edges with Dirichlet boundary conditions on one endpoint, the number of rows of $A_\text{diff}$ with non-null sum is equivalent to the number of such edges $|\mathcal{N}_s|$, which was assumed to be non-null.\\
        Finally, since $A_\text{diff}$ is symmetric (Property~\ref{property:B:symmetric}), the same holds by columns.
	\end{proof}

	\noindent A consequence of Property~\ref{property:B:row_sum} concerns the diagonal dominance of the matrix $A_\text{diff}$:
	
	\begin{property}
		The matrix $A_\text{diff}$ defined in equation~\eqref{eq:diffusion:matrix:B} is diagonally dominant on every row and strictly diagonally dominant on a number $n=|\mathcal{N}_s|$ of rows.
	\label{property:B:diagonally_dominant}
	\end{property}

	\noindent These rows correspond to edges where Dirichlet boundary conditions are imposed.
	
	In order to prove existence and uniqueness of the solution and positivity of the discrete diffusion problem, we also need to introduce the \textit{SC property}, and show that it is satisfied by $A_\text{diff}$.
	
	\begin{property}
		The $(n_e+n_\text{bif})\times(n_e+n_\text{bif})$ square matrix $A_\text{diff}$ defined in \eqref{eq:diffusion:matrix:B} satisfies the \textbf{SC property} (Definition 6.2.7 of \cite{horn2012matrix}):\\
		$\forall p,q\in\{1,\dots,n_e+n_\text{bif}\}$, with $p\neq q$, there is a sequence of distinct integers $\{k_i\}_{i=1}^{m}\subseteq\{1,\dots,n_e+n_\text{bif}\}$, such that $k_1 = p, \ k_m = q$ and $(A_\text{diff})_{k_1 k_2},\ (A_\text{diff})_{k_2 k_3}, \ \dots,\ (A_\text{diff})_{k_{m-1} k_m}$ are all nonzeros.
		\label{property:B:SC}
	\end{property}

	\begin{proof}
		Le us analyze the non-zero entries of the matrix:
        $(A_\text{diff})_{ik}\neq 0 \Leftrightarrow k=i$ or $k$ and $i$ represent neighboring edges or $k$ and $i$ represent a bifurcation and one of the edges connected to it.\\
		As a consequence, $\exists k_1,\dots,k_m\in\{1,\dots,n_e+n_\text{bif}\}$ distinct indices such that $(A_\text{diff})_{k_1 k_2},\ (A_\text{diff})_{k_2 k_3},\\ \dots,\ (A_\text{diff})_{k_{m-1} k_m}\ \neq 0 \ \Leftrightarrow \exists$ a path on the 1D domain $\Lambda$ connecting the edge or bifurcation $k_1$ to the edge or bifurcation $k_m$ without passing more than once through the same edge or bifurcation. This condition is true for every couple of edges/bifurcation points because the domain $\Lambda$ is a completely connected graph by construction. \\
		Thus, we can conclude that $A_\text{diff}$ satisfies the SC property.
	\end{proof}

	\noindent Thanks to these properties, we can prove the positivity of the numerical scheme proposed in Section~\ref{section:NM:Diffusion} and the existence and uniqueness of the solution to the discrete problem.

	\begin{theorem}
		Let $A_\text{diff}$ be the $(n_e+n_\text{bif})\times(n_e+n_\text{bif})$ matrix defined in equation~\eqref{diffusion:B}. Then, $A_\text{diff}$ is invertible and the problem \eqref{eq:diffusion:matrix} is positive.
        \label{theorem:diffusion}
	\end{theorem}
	
	\begin{proof} 
		Since the matrix $A_\text{diff}$ satisfies the SC Property~\ref{property:B:SC} and Property~\ref{property:B:diagonally_dominant} ensures that it is also diagonally dominant on each row and strictly diagonally dominant on at least one row, then it is invertible (see Better Corollary 6.2.9 of \cite{horn2012matrix}).

         Moreover, these conditions are also satisfied by $A_\text{diff}+\tilde{D}$, where $\tilde{D}$ denotes a general diagonal $(n_e+n_\text{bif})\times(n_e+n_\text{bif})$ matrix $\tilde{D}$ with all positive diagonal entries. Thus, $A_\text{diff}+\tilde{D}$ is non-singular $\forall\tilde{D}$, by Better Corollary 6.2.9 of~\cite{horn2012matrix}, and consequently, since $A_\text{diff}$ is a Z-matrix (Property~\ref{property:B:Zmatrix}), this implies (Theorem 1 of~\cite{PLEMMONS1977175}) that $A_\text{diff}$ is a nonsignular M-matrix and also inverse-positive.\\
		This means that $(A_\text{diff}^{-1})_{ik}\geq 0 \ \forall i,k=1,\dots,n_e+n_\text{bif}$ and $A_\text{diff}^{-1}\neq 0$. Then, $\mathbf{u}=A_\text{diff}^{-1}\mathbf{g}\geq 0$, $\forall \mathbf{g}\geq 0$.
	\end{proof}
	
	\noindent Let us now explicitly write the system for the vectors $\mathbf{\hat{u}}$ and $\mathbf{u}_B$ of unknowns on the edges and on the bifurcation points, respectively, starting from the discrete problem~\eqref{eq:diffusion:matrix} in matrix form:
	
	\begin{equation*}
		\begin{cases}
			B\mathbf{\hat{u}} + C\mathbf{{u}_B} = \mathbf{f}, \\
			C^T \mathbf{\hat{u}} + D\mathbf{{u}_B} =0.
		\end{cases}
	\end{equation*}

	\noindent Since $D$ is invertible (diagonal matrix with all positive entries), we can solve the second equation for $\mathbf{u}_B$ and obtain the following system:

	\begin{equation*}
		\begin{cases}
			(B-CD^{-1}C^T)\mathbf{\hat{u}} =\mathbf{f}, \\
			\mathbf{{u}_B} = -D^{-1}C^T\mathbf{\hat{u}}.
		\end{cases}
	\end{equation*}

	\noindent Observe that $B-CD^{-1}C^T$ is nothing but the Schur complement of $D$, which we denote by $A/D$:

	\begin{equation*}
		\begin{cases}
			(A/D)\mathbf{\hat{u}} = \mathbf{f}, \\
			\mathbf{{u}_B} = -D^{-1}C^T\mathbf{\hat{u}}.
		\end{cases}
	\end{equation*}
	
	\noindent Since $A$ and $D$ are nonsingular, we know that $A/D$ is nonsingular as well (Theorem 1.2 of \cite{zhang2006schur}) and that $\det(A/D)=\dfrac{\det(A)}{\det(D)}$, by Schur's formula.
	
	Our goal is now to prove that also the Schur complement $A/D$ is invertible and positive. To this aim, we need to recall the definition of the inertia of an Hermitian matrix and an associated result (see~\cite{zhang2006schur}):
		
	\begin{definition}
		We call \textbf{inertia} of an Hermitian matrix $A$ the triplet $In(A):=(p(A),q(A),z(A))$, where $p(A)$ is the number of positive eigenvalues of $A$, $q(A)$ the number of negative eigenvalues of $A$ and $z(A)$ the multiplicity of the 0 eigenvalue.
	\end{definition}
	
	\begin{theorem}
		Theorem 1.6 of \cite{zhang2006schur}:\\
		Let $A$ be an Hermitian matrix and $A_{11}$ a nonsingular principal submatrix of $A$. Then,
		\begin{equation}\nonumber
			In(A)=In(A_{11})+In(A/A_{11}).
		\end{equation}
	\end{theorem}
	
	\noindent We can now prove the following result:
	
	\begin{theorem}
		Let $A$ be the $(n_e+ n_\text{bif})\times(n_e+ n_\text{bif})$ matrix defined in equation~\eqref{eq:diffusion:matrix:B} and $A/D$ the Schur complement of its south-east $n_\text{bif}\times n_\text{bif}$ block $D$. Then, $A/D$ is invertible and the system $(A/D)\mathbf{u}=\mathbf{f}$ is positive.
		\label{theorem:diffusion:shur}
	\end{theorem}
	
	\begin{proof}
		The matrices $A$ and $D$ are both positive definite, because they have positive inverse and a diagonal with strictly positive entries. Then, $In(A)=(n_e\times n_\text{bif}, 0, 0)$, $In(D)=(n_\text{bif},0,0)$ and $In(A/D)=In(A)-In(D)=(n_e,0,0)$, by the previous theorem. Then, $A/D$ only has positive eigenvalues, and therefore it is positive definite.
		
		Let us now inspect the entries of the matrix $CD^{-1}C^T$.\\
		The non-diagonal ones are positive because they are given sums of non-negative quantities:
		\begin{equation*}
				(CD^{-1}C^T)_{ik}=\sum_{h=1}^{n_\text{bif}} C_{ih}\dfrac{C_{kh}}{D_{hh}}\geq 0 \ \forall i,k = 1,\dots, n_e,\ i\neq k,
		\end{equation*}
		\noindent and the same holds for diagonal ones:
		\begin{equation*}
				(CD^{-1}C^T)_{ii} = \sum_{h=1}^{n_\text{bif}}\dfrac{C_{ih}^2}{D_{hh}}
				\begin{cases}
					>0 & \text{if\ } i \text{\ is\ connected\ to\ a\ bifurcation}, \\
					0 & \text{otherwise};
				\end{cases}
		\end{equation*}
		\noindent indeed, $(D^{-1}C^T)_{ik}=\dfrac{1}{D_{ii}}(C^T)_{ik}=\dfrac{C_{ki}}{D_{ii}}\leq 0\ \forall i=1,\dots,n_\text{bif},\ k=1,\dots,n_e$ and $C_{ih}\leq 0 \ \forall i=1,\dots,n_\text{bif},\ k=1,\dots,n_e$.
		
		\noindent Then,

        \begin{equation*}
            (B-CD^{-1}C^T)_{ik} =
            B_{ik} - (CD^{-1}C^T)_{ik}
            \leq 0, \quad \forall i,k=1,\dots,n_e,\ i\neq k, 
        \end{equation*}

        \noindent because $B_{ik}\leq 0$ and $(CD^{-1}C^T)\geq 0$, $\forall i,k=1,\dots,n_e,\ i\neq k$, and

        \begin{multline*}
            (B-CD^{-1}C^T)_{ii} =
            B_{ii} - (CD^{-1}C^T)_{ii} \geq
            -\sum_{\scriptsize
            \begin{split}
                j=1\\
                j\neq i
            \end{split}
            }^{n_e} B_{i,j} - \sum_{j=1}^{n_\text{bif}} C_{i,j}
            - \sum_{j=1}^{n_\text{bif}} \dfrac{C_{i,j}^2}{D_{jj}} = \\
            = -\sum_{\scriptsize
            \begin{split}
                j=1\\
                j\neq i
            \end{split}
            }^{n_e} B_{i,j}
            - \sum_{j=1}^{n_\text{bif}} C_{ij} \dfrac{D_{jj} - C_{i,j}}{D_{jj}}
            >0, \quad \forall i=1,\dots,n_e,
        \end{multline*}

        \noindent because $\sum_{j=1,\ j\neq i}^{n_e} B_{i,j} \leq0$, $\forall i\in\{1,\dots,n_e\}$, and $D_{jj} - C_{i,j} > 0$, $C_{ij} < 0$ and $D_{jj} > 0$, $\forall i\in\{1,\dots,n_e\}$, $\forall j\in\{1,\dots,n_\text{bif}\}$.

		
		Consequently, $A/D=B-CD^{-1}C^T$ is a Z-matrix and has positive diagonal entries. Thus, the following conditions are equivalent (Theorem 1 of \cite{PLEMMONS1977175}):
		
		\begin{itemize}
			\item $A/D$ has all positive eigenvalues;
			\item $A/D$ is a nonsingular M-matrix;
			\item $A/D$ is inverse-positive. 
		\end{itemize}
		
    \noindent As a consequence, $A/D$ is invertible and, since it is inverse-positive, $\mathbf{u}=(A/D)^{-1}\mathbf{f}\geq 0$, i.e. the system $(A/D)\mathbf{u}=\mathbf{f}$ is positive.	
	\end{proof}
	\noindent Finally, we can prove that the flux approximation is consistent at graph nodes by following a similar procedure as in Theorem 7.1 of \cite{eymard2000finite} for the  upwind flux with TPFA on the diffusion term.
 
	\begin{theorem}
	If the diffusion coefficient is globally continuous on $\Lambda$ and $\nu_k\in\mathcal{C}^1(e_k), \ \forall k\in\{1,\dots,\mathcal{N}_e\}$, then the flux approximation~\eqref{eq:diffusion:flux_approx} is consistent at every internal node of $\Lambda$.
	\label{theorem:diff:consistency}
	\end{theorem}
	
	\begin{proof}
	Define the numerical flux entering a node $v$ from the incoming edge $e_k$ by
	
	\begin{equation}
		\tilde{F}_{v,k}^+ = -\nu_{k,v}\dfrac{2}{|e_k|}\left(u_v - u_k\right),
		\label{eq:th:diff_cons:Fin}
	\end{equation}
	
	\noindent where $u_v$ and $u_k$ are the approximate solutions at the node $v$ and on $e_k$, respectively, and $\nu_{k,v}$ the diffusion constant evaluated at the node.\\
	Similarly, let
	
	\begin{equation*}
		\tilde{F}_{v,k}^- = -\nu_{k,v}\dfrac{2}{|e_k|}\left(u_k - u_v\right)
	\end{equation*}
	
	\noindent be the numerical flux going from the node $v$ into the outgoing edge $e_k$.\\
	Then, the total incoming and outgoing fluxes at the node $v$ are $\tilde{F}_v^+ = \sum_{ e_k\in\mathcal{E}_v^+} \tilde{F}_{v,k}^+$ and $\tilde{F}_v^- = \sum_{ e_k\in\mathcal{E}_v^-} \tilde{F}_{v,k}^-$.\\
	

	
	

	
	\noindent Let us now introduce the exact flux of the velocity on an edge $e_k$ evaluated at the node on $s_v\in\Lambda$:
	
	\begin{equation*}
		F_{v,k} = -\left(\nu \dfrac{du\big|_{e_k}}{ds}\right) (s_v)
	\end{equation*}
	
	\noindent and the evaluation of the numerical flux with respect to the exact solution $u$:

    \begin{equation}
        F^\star_v = \sum_{ e_k\in\mathcal{E}_v^+} F^{\star,+}_k = \sum_{ e_k\in\mathcal{E}_v^-} F^{\star,-}_k,
        \label{eq:th:Fstar_balance}
    \end{equation}
 
	
    \noindent where $F^{\star,+}_k$ and $F^{\star,-}_k$ are the evaluations of the numerical fluxes on $e_k$, incoming and outgoing edge, respectively, with respect to the exact solution:
	
	\begin{equation}
		F^{\star,+}_k = -\nu_{k,v}\dfrac{2}{|e_k|}\left(u\big|_{e_k}(s_v) - u\big|_{e_k}(s_k)\right),
		\label{eq:th:diff:Finstar}
	\end{equation}
	\begin{equation*}
		F^{\star,-}_k = -\nu_{k,v}\dfrac{2}{|e_k|}\left(u\big|_{e_k}(s_k) - u\big|_{e_k}(s_v)\right).
        \label{eq:th:diff:Foutstart}
	\end{equation*}
	
	\noindent We want to show that the truncation error $\tau_v \coloneq F^\star_v - F_v$ is controlled by the maximum length of the segments discretizing $\Lambda$, $h=\max_{k\in\mathcal{E}}|e_k|$, i.e.
	
	\begin{equation}
		\exists C\in\mathbb{R}_+^\star \text{\ such\ that\ } |\tau_v|\leq Ch.
		\label{eq:th:diff:th}
	\end{equation}
	
	\noindent We start by observing that, since $\nu\in\mathcal{C}^1(e_k),\ \forall e_k\in\mathcal{E}_v$, then $u\in\mathcal{C}^2(e_k)$ and, therefore, we can write the second order Taylor expansion of $u$ on each edge as $u(s_k) = u(s_v) + \dfrac{|e_k|}{2}\dfrac{du}{ds}(s_k) + R_k$, with $R_k \leq C_k |e_k|\leq C_k h,\ C_k > 0$. Then, substituting it in~\eqref{eq:th:diff:Finstar} and~\eqref{eq:th:diff:Foutstart}, $\forall k\in\{1,\dots,\mathcal{N}_e\}\ \exists C_k\in\mathbb{R}_+^\star$ such that
	
	\begin{equation}
		\begin{cases}
			F^{\star,+}_k = \omega_k F_v + R^+_k, \ &|R^+_k| \leq C_k h, \ \forall e_k\in\mathcal{E}_v^+,\\
			F^{\star,-}_j = -\omega_j F_v + R^-_j, \ &|R^-_j| \leq C_j h, \ \forall e_j\in\mathcal{E}_v^-,
		\end{cases}
		\label{eq:th:diff:Flinear}
	\end{equation}
	
	\noindent where $\omega_kF_v$ denotes the fraction of flux $F_v$ through the node $v$ coming from each edge $e_k$, and is given by: 

    \begin{equation*}
            \omega_k = \dfrac{\nu_{k,v}}{\nu(s_v)}, \ e_k\in\mathcal{E}_v.
    \end{equation*}
    
    \noindent As a consequence of~\eqref{eq:th:Fstar_balance} and~\eqref{eq:th:diff:Flinear}, we have that

    \begin{equation*}
        \sum_{e_k\in\mathcal{E}_k^+} \omega_k F_v + R^+ = \sum_{e_k\in\mathcal{E}_k^-} \omega_k F_v + R^-,
    \end{equation*}

    \noindent with $R^+ = \sum_{ e_k\in\mathcal{E}_v^+} R^+_k $ and $R^- = \sum_{ e_k\in\mathcal{E}_v^-} R^-_k$.

    \noindent Thus, we obtain

    \begin{equation*}
        F_v = \dfrac{R^- - R^+}{\sum_{k\in\mathcal{E}_v}\omega_k} = \dfrac{R^- - R^+}{\sum_{k\in\mathcal{E}_v}\dfrac{\nu_{k,v}}{\nu(s_v)}}.
    \end{equation*}

    \noindent Assuming that the diffusion coefficient $\nu$ is continuous, its evaluation on each edge $e_k$ at the common endpoint must be the same as at the node, i.e. $\nu_{k,v} = \nu(s_v), \ \forall e_k\in\mathcal{E}_v$. Then, denoting by $|\mathcal{E}_v|$ the number of segments intersecting at the node $v$, we have that $\sum_{e_k \in \mathcal{E}_v} \nu_{k,v} = |\mathcal{E}_v|\nu(s_v)$ and $\sum_{e_k\in\mathcal{E}_v} \omega_k = |\mathcal{E}_v|$. Consequently $F_v = \dfrac{R^- - R^+}{|\mathcal{E}_v|} $,
	and the truncation error is given by

    \begin{equation*}
        |\tau_v| = \left| F^\star_v - F_v \right| = \left| \left(|\mathcal{E}_v^+| -1 \right) F_v + R^+ \right| = \left| \left(|\mathcal{E}_v^+| -1 \right) \dfrac{R^- - R^+}{|\mathcal{E}_v|} + R^+ \right|.
    \end{equation*}
 

    \noindent By triangular inequality and recalling that every internal node has at least one incoming and one outgoing edge, i.e. $|\mathcal{E}_v|\geq |\mathcal{E}_v^+|\geq 1$,

    \begin{equation*}
        |\tau_v| \leq \left(\dfrac{|\mathcal{E}_v^+|}{|\mathcal{E}_v|} + \dfrac{1}{|\mathcal{E}_v|} \right) \left(|R^-| + |R^+|\right) + |R^+| \leq 2 \left(|R^-| + |R^+|\right) + |R^+| \leq (2C^- + 3C^+)h = Ch,
    \end{equation*}
 
	
	
	
	
	
	
	\noindent thus proving that \eqref{eq:th:diff:th} holds.
	
	\end{proof}

	\subsection{Drift-diffusion equation}
	\label{section:NA:DiffusionTransport}
	In Section~\ref{section:NM:DiffusionTransport} we have introduced a numerical scheme for the advection-diffusion equation~\eqref{eq:diffusion_transport} and obtained a linear system of equation $M\mathbf{u}=\mathbf{f}$. The matrix $M$ is given by a sum of block matrices resulting from the discretizations of the pure transport and diffusion problems, analyzed above. Then, based on the results presented for the two separate problems, we can prove the same results for this problem, as well. In particular, since $M$ is given by the sum of the two Z-matrices $M_\text{tr}$ and $M_\text{diff}$,
	
	\begin{property}
		The matrix $M$ defined in equation~\eqref{eq:diffusion_transport:matrix:M} is a Z-matrix.
		\label{property:M:Zmatrix}
	\end{property}
	
	\noindent Thanks to the construction of $M$, its resulting structure is very similar to those of $M_\text{tr}$ and $A_\text{diff}$, and thus we can prove the desired properties similarly to what we have done for the two separate problems.\\
    Since the set of non-zero elements in $M_\text{tr}$ is a subset of the non-zero elements of the top-left block $B$ of $A_\text{diff}$, the sum of these two matrices still satisfies the SC property. Moreover, transposing them does not affect their pattern, and as a consequence also $M^T$ satisfies the SC property.

    \begin{property}
        The matrix $M$ defined in equation~\eqref{eq:diffusion_transport:matrix:M} and its transpose $M^T$ satisfy the \textbf{SC property} (Definition 6.2.7 of \cite{horn2012matrix}).
        \label{property:M:sc}
    \end{property}

    \noindent Moreover, as a consequence of Properties~\ref{property:A:colsum},~\ref{property:B:row_sum} and~\ref{property:B:symmetric}, we know that $M$ has non-negative column sums:

    \begin{property}
        The row sums of the elements of the matrix $M$ defined in equation~\eqref{eq:diffusion_transport:matrix:M} are always zero but fot a number $n=|\mathcal{N}_s|>0$ of rows.
        \label{property:M:colsum}
    \end{property}

    \noindent As a consequence of Properties~\ref{property:B:SC},~\ref{property:M:sc} and~\ref{property:M:colsum}, we can show, as in the proof of Theorem~\ref{theorem:diffusion}, that the transpose $M^T$ of the matrix $M$ is inverse-positive and, therefore, also $M$ is. Then, the following result holds:
 
	\begin{theorem}
		The matrix $M$ defined in equation~\eqref{eq:diffusion_transport:matrix:M} is non-singular and the associated discrete drift-diffusion scheme is positive.
	\end{theorem}

		\noindent Finally, we can apply the same procedure adopted for the diffusion equation in the previous section, explicitly write the system \eqref{eq:diffusion_transport:matrix:M} for the unknown $\mathbf{\hat{u}}$ as follows:
		
		\begin{equation*}
		\begin{cases}
			(M/D_t) \mathbf{\hat{u}} = \mathbf{f}, \\
			\mathbf{u_B} = -D^{-1}C^T \mathbf{\hat{u}},
		\end{cases}
		\end{equation*}

		\noindent where we have denoted by $D_t$ the bottom-right block of the drift-diffusion matrix: $D_t = \Delta t D$ and by $M/D_t$ its Schur complement. Following the same steps as in the proof of Theorem~\ref{theorem:diffusion:shur}, we can prove the following result:
	
	\begin{theorem}
		Let $M$ be the $(n_e+ n_\text{diff})\times(n_e+ n_\text{diff})$ matrix defined by \eqref{eq:diffusion_transport:matrix:M} and $M/D_t$ the Schur complement of its south-east diagonal block $D_t$. Then, $M/D_t$ is invertible and the system \eqref{eq:diffusion_transport:matrix:M} is positive.
	\end{theorem}
	
	\section{Results}
	\label{section:Results}
	In this section we show some applications of the proposed methods to different geometries. 
    We first check that the approximation errors on one-dimensional straight lines respect the well-known convergence results of Finite Volumes, then we study the convergence on domains with bifurcations and finally apply the methods to a drift-diffusion problem on the complex geometry of an electrical tree.
	
	The error is computed as the $L^1$ norm of the difference between the approximated and the exact solutions at the final time $T$, normalized with respect to the $L^1$ norm of the exact solution:
	
	\begin{equation*}
		\text{error} = \dfrac{\|u_\text{ex}(\cdot,T) - \tilde{u}^{N_t}\|_{L^1([0,1])}}{\|u_\text{ex}(\cdot,T)\|_{L^1([0,1])}},
	\end{equation*}

	\noindent where we denote by $\tilde{u}^l$ the approximate solution at each discrete time $t_l$ as a piecewise constant function, taking values $\tilde{u}_k^l$ on $e_k, \ k=1,\dots,n_e$, $l=1,\dots,N_t$.
	
	\subsection{TC1: Transport equation on a straight line}
	\label{section:R:ConvergenceTransport}
	The discretization presented in Section~\ref{section:NM:Transport} of the transport equation on a line is nothing but the Finite Volume upwind scheme, see~\cite{leveque1992numerical}.
	

	We apply it to solve the homogeneous transport equation~\eqref{eq:transport} with null source term $f=0$ and advection speed $c=0.5$, on the space-time domain $[0,1]\times[0,1]$, with initial condition $u_0 (x) = 0,\ \forall x\in[0,1]$ and inflow boundary condition $u(0,t) = 1,\ \forall t\in(0,1]$.\\
	The exact solution of this problem consists in the propagation of a rectangular wave of speed $c$:
	
	\begin{equation*}
		u_\text{ex}(x,t) =
		\begin{cases}
			1, & x\in[0,ct], \ t\in(0,1], \\
			0, & x\in(ct,1], \ t\in(0,1].
		\end{cases}
		\tag{TC1A}
        \label{eq:tc1a}
	\end{equation*}
	
	\noindent In the plots in Figure~\ref{plot:line:transport} we can observe that the convergence rate in space and time of $\frac{1}{2}$, which is the expected value for drift equations with discontinuous solution, is recovered by our method.
	
	Convergence of order 1 is expected when the solution is continuous, and we can appreciate it in Figure~\ref{plot:line:transport:cont}, where we have applied the proposed method to an homogeneous transport equation with advection speed $c=0.5$ with initial condition
	
	\begin{equation*}
		u_0(x) = \sin(\pi x), \quad x\in[0,1],
	\end{equation*}
	
	\noindent whose exact solution is
	
	\begin{equation*}
		u_\text{ex}(x,t) = u_0(x-ct), \quad \forall x\in[0,1],\ t \in [0,1].
		\tag{TC1B}
        \label{eq:tc1b}
	\end{equation*}
	
	\noindent We have imposed a Dirichlet condition on the inflow boundary accordingly: $u(0,t) = u_0(-ct), \ \forall t\in[0,1]$.
	
	In both cases we observe that the error in space saturates for small values of $h$ if the time grid is not refined enough and viceversa. This is due to the predominance of the error made on the coarsest grid. In the discontinuous~\eqref{eq:tc1a} case, (Figure~\ref{plot:line:transport}) the two plots have almost the same behavior and both errors saturate for comparable mesh spacings, while in the continuous case (Figure~\ref{plot:line:transport:cont}) the error in space needs a finer time discretization to avoid saturation, indicating a predominance of the error committed by the discretization in space over that in time in problem~\eqref{eq:tc1b}. This is due to the diffusivity of the implicit time scheme, that suffers in presence of discontinuities.

	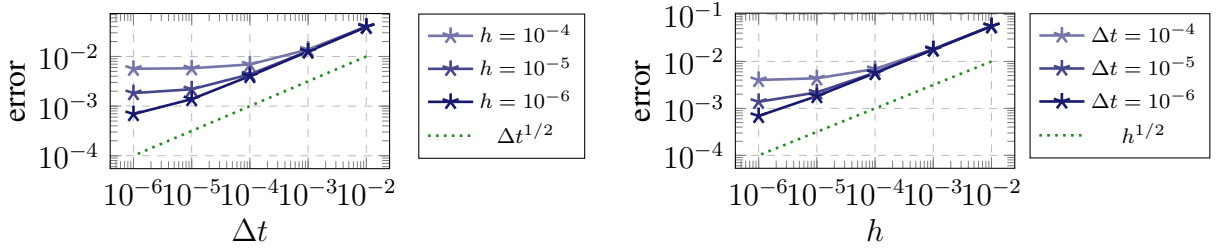
\begin{figure}
		\centering
		\subfloat[\centering Error with respect to time discretization.]
		{\begin{tikzpicture}
		\begin{loglogaxis}
			[scale = .5,
			width = .45\textwidth,
			height = .25\textwidth,
			xlabel={$\Delta t$},
			ylabel={error},
			scale only axis,
			xtick={},
			ytick={},
			ymajorgrids=true,
			xmajorgrids=true,
			grid style=dashed,
			legend style = {draw=black, font=\scriptsize, at={(1.7,1)}},
			legend entries ={
				 $h =10^{-4}$, $h =10^{-5}$, $h=10^{-6}$, 
				 $\Delta t^{1/2}$},
			]
			\addplot[color=darkBlue!60, mark=star, mark size =3.2, line width = 1]
			coordinates
			{(1e-2,0.0402572)
			(1e-3,0.0138184)
			(1e-4,0.00690975)
			(1e-5,0.00578111)
			(1e-6,0.00565589)};
		\addplot[color=darkBlue!80, mark=star, mark size =3.2, line width = 1]
			coordinates
			{(1e-2,0.039857)
			(1e-3,0.0127397)
			(1e-4,0.00437014)
			(1e-5,0.0021851)
			(1e-6,0.00182819)};
		\addplot[color=darkBlue, mark=star, mark size =3.2, line width = 1]
			coordinates
			{(1e-2,0.0394009)
			(1e-3,0.0126067)
			(1e-4,0.00402888)
			(1e-5,0.00138197)
			(1e-6,0.000690989)};
		\addplot[color=lightGreen, dotted, line width = 1]
			coordinates
			{(1e-2,1e-2)
				(1e-6,1e-4)};
		\end{loglogaxis}
		\end{tikzpicture}
		\quad}	
		\subfloat[\centering Error with respect to space discretization.]
		{\begin{tikzpicture}
			\begin{loglogaxis}
				[scale = .5,
				width = .45\textwidth,
				height = .25\textwidth,
				xlabel={$h$},
				ylabel={error},
				scale only axis,
				xtick={},
				ytick={},
				ymajorgrids=true,
				xmajorgrids=true,
				grid style=dashed,
				legend style = {draw=black, font=\scriptsize, at={(1.7,1)}},
				legend entries ={
					$\Delta t = 10^{-4}$, $\Delta t = 10^{-5}$, $\Delta t = 10^{-6}$, 
					$h^{1/2}$},
				]
				\addplot[color=darkBlue!60, mark=star, mark size =3.2, line width = 1]
				coordinates
					{(1e-2,0.0564656)
					(1e-3,0.0182788)
					(1e-4,0.00690975)
					(1e-5,0.00437014)
					(1e-6,0.00402888)};
				\addplot[color=darkBlue!80, mark=star, mark size =3.2, line width = 1]
				coordinates
					{(1e-2,0.0563389)
					(1e-3,0.0182788)
					(1e-4,0.00578111)
					(1e-5,0.0021851)
					(1e-6,0.00138197)};
				\addplot[color=darkBlue, mark=star, mark size =3.2, line width = 1]
				coordinates
					{(1e-2,0.0563264)
					(1e-3,0.0178427)
					(1e-4,0.00565589)
					(1e-5,0.00182819)
					(1e-6,0.000690989)};
				\addplot[color=lightGreen, dotted, line width = 1]
				coordinates
				{(1e-2,1e-2)
					(1e-6,1e-4)};
			\end{loglogaxis}
		\end{tikzpicture}
		}	
		\caption{\textbf{TC1A -} Convergence test for the transport equation on a line with discontinuous exact solution. Normalized error in norm $L^1([0,1])$ computed at time $t=1$ corresponding to space and time meshes with different dimensions.}
		\label{plot:line:transport}
	\end{figure}

	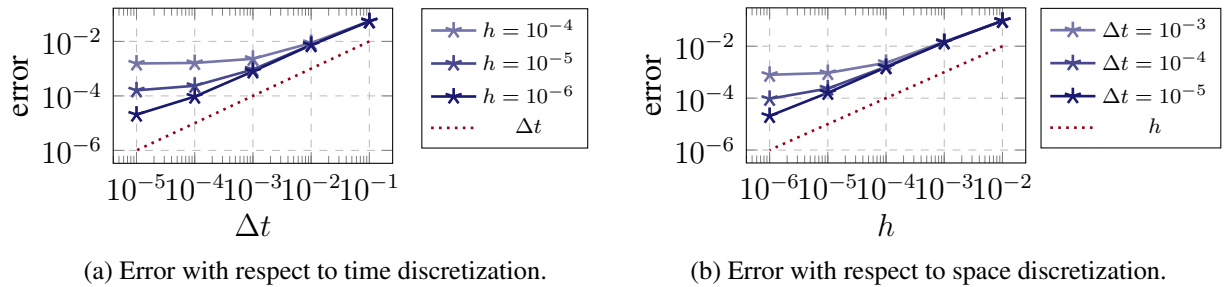
\begin{figure}
		\centering
		\subfloat[\centering Error with respect to time discretization.]
		{\begin{tikzpicture}
			\begin{loglogaxis}
				[scale = .5,
				width = .45\textwidth,
				height = .25\textwidth,
				xlabel={$\Delta t$},
				ylabel={error},
				scale only axis,
				xtick={},
				ytick={},
				ymajorgrids=true,
				xmajorgrids=true,
				grid style=dashed,
				legend style = {draw=black, font=\scriptsize, at={(1.7,1)}},
				legend entries ={
					 $h =10^{-4}$, $h =10^{-5}$, $h =10^{-6}$, $\Delta t$},
				]
				\addplot[color=darkBlue!60, mark=star, mark size = 3.2,
				line width = 1]
				coordinates
				{
					(1e-1, 0.0554731)
					(1e-2, 0.00874119)
					(1e-3, 0.00229917)
					(1e-4, 0.00161976)
					(1e-5, 0.00154793)
				};			
				\addplot[color=darkBlue!80, mark=star, mark size = 3.2,
				line width = 1]
				coordinates
				{
					(1e-1, 0.0546917)
					(1e-2, 0.00749042)
					(1e-3, 0.0009313)
					(1e-4, 0.000234759)
					(1e-5, 0.000160177)
				};	
				\addplot[color=darkBlue, mark=star, mark size = 3.2,
				line width = 1]
				coordinates
				{
					(1e-1,0.0546133)
					(1e-2,0.00736436)
					(1e-3,0.000792741)
					(1e-4,9.40881e-05)
					(1e-5,2.05956e-05)
};
				\addplot[color=darkRed, dotted, line width = 1]
				coordinates
				{(1e-1,1e-2)
					(1e-5,1e-6)};
			\end{loglogaxis}
		\end{tikzpicture}
		\quad}
		\subfloat[\centering Error with respect to space discretization.]
{		\begin{tikzpicture}
			\begin{loglogaxis}
				[scale = .5,
				width = .45\textwidth,
				height = .25\textwidth,
				xlabel={$h$},
				ylabel={error},
				scale only axis,
				xtick={},
				ytick={},
				ymajorgrids=true,
				xmajorgrids=true,
				grid style=dashed,
				legend style = {draw=black, font=\scriptsize, at={(1.7,1)}},
				legend entries ={
					$\Delta t = 10^{-3}$, $\Delta t = 10^{-4}$, $\Delta t = 10^{-5}$, 
					$h$},
				]
				\addplot[color=darkBlue!60, mark=star, mark size = 3.2,
				line width = 1]
				coordinates
				{(1e-2,0.0955992)
				(1e-3,0.0148219)
				(1e-4,0.00229917)
				(1e-5,0.0009313)
				(1e-6,0.000792741)
				};
				\addplot[color=darkBlue!80, mark=star, mark size =3.2, line width = 1]
				coordinates
				{(1e-2,0.0953186)
				(1e-3,0.0142374)
				(1e-4,0.00161976)
				(1e-5,0.000234759)
				(1e-6,9.40881e-05)
				};
				\addplot[color=darkBlue, mark=star, mark size =3.2, line width = 1]
				coordinates
				{(1e-2,0.0952913)
				(1e-3,0.0141771)
				(1e-4,0.00154793)
				(1e-5,0.000160177)
				(1e-6,2.05956e-05)
				};
				\addplot[color=darkRed, dotted, line width = 1]
				coordinates
				{(1e-2,1e-2)
					(1e-6,1e-6)};
			\end{loglogaxis}
		\end{tikzpicture}
		}
	\caption{\textbf{TC1B -} Convergence test for  the transport equation on a line with continuous exact solution. Normalized error in norm $L^1([0,1])$ computed at time $t=1$ corresponding to space and time meshes with different dimensions.}
	\label{plot:line:transport:cont}
	\end{figure}
	
	\subsection{TC2: Diffusion equation on a straight line}
	\label{section:R:ConvergenceDiffusion}
    The method is applied to the solution of equation~\eqref{eq:diffusion_transport} with $c=0$, constant diffusion coefficient $\nu=2$ and initial condition $u_0(x) = \sin(\pi x), \ x\in [0,1]$. We want to approximate the exact solution
	
	\begin{equation*}
		u_\text{ex}(x,t) = \sin(\pi x) e^{-2\pi^2 t}, \quad x \in [0,1], t \in [0,1],
		\tag{TC2}
        \label{eq:tc2}
	\end{equation*}

	\noindent and impose Dirichlet boundary conditions at the endpoints accordingly: $u_\text{ex}(0,t) = u_\text{ex}(1,t) = 0, \ \forall t\in(0,1]$.
	
	In Figure~\ref{plot:line:diffusion} we can observe that the theoretical orders of convergence for the Implicit Euler and FV-TPFA, 1 and 2 respectively, are verified. The error saturates for very small sizes of the spatial mesh, due to the predominance of the error induced by the time discretization.
	
		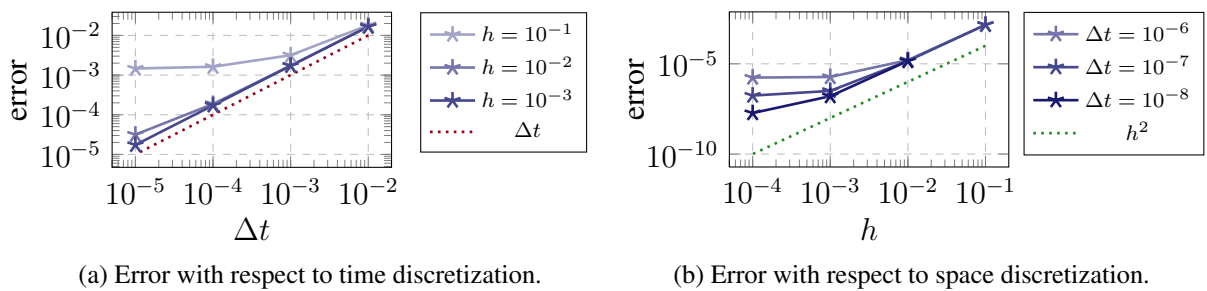
\begin{figure}
		\centering
		\subfloat[\centering Error with respect to time discretization.]
		{\begin{tikzpicture}
			\begin{loglogaxis}
				[scale = .5,
				width = .45\textwidth,
				height = .25\textwidth,
				xlabel={$\Delta t$},
				ylabel={error},
				scale only axis,
				xtick={},
				ytick={},
				ymajorgrids=true,
				xmajorgrids=true,
				grid style=dashed,
				legend style = {draw=black, font=\scriptsize, at={(1.7,1)}},
				legend entries ={$h = 10^{-1}$, $h =10^{-2}$, $h = 10^{-3}$,
					$\Delta t$},
				]
				\addplot[color=darkBlue!40, mark=star, mark size = 3.2,
				line width = 1]
				coordinates
				{(1e-2,0.0181507)
					(1e-3,0.00317205)
					(1e-4,0.00162147)
					(1e-5,0.00146587)
				};
				\addplot[color=darkBlue!60, mark=star, mark size = 3.2,
				line width = 1]
				coordinates
				{(1e-2,0.0166604)
					(1e-3,0.00173138)
					(1e-4,0.000186591)
					(1e-5,3.15869e-05)
				};
				\addplot[color=darkBlue!80, mark=star, mark size =3.2, line width = 1]
				coordinates
				{(1e-2,0.0166457)
					(1e-3,0.0017171)
					(1e-4,0.000172369)
					(1e-5,1.73715e-05)
				};
				\addplot[color=darkRed, dotted, line width = 1]
				coordinates
				{(1e-2,1e-2)
					(1e-5,1e-5)};
			\end{loglogaxis}
		\end{tikzpicture}\quad}
%
		\subfloat[\centering Error with respect to space discretization.]
		{\begin{tikzpicture}
			\begin{loglogaxis}
				[scale = .5,
				width = .45\textwidth,
				height = .25\textwidth,
				xlabel={$h$},
				ylabel={error},
				scale only axis,
				xtick={},
				ytick={},
				ymajorgrids=true,
				xmajorgrids=true,
				grid style=dashed,
				legend style = {draw=black, font=\scriptsize, at={(1.7,1)}},
				legend entries ={
					$\Delta t = 10^{-6}$, $\Delta t = 10^{-7}$, $\Delta t = 10^{-8}$, 
					$h^2$},
				]
				\addplot[color=darkBlue!60, mark=star, mark size =3.2, line width = 1]
				coordinates
				{(1e-1,0.00145031)
					(1e-2,1.60812e-05)
					(1e-3,1.8661e-06)
					(1e-4,1.72427e-06)
};
				\addplot[color=darkBlue!80, mark=star, mark size =3.2, line width = 1]
				coordinates
				{(1e-1,0.00144875)
					(1e-2,1.45306e-05)
					(1e-3,3.15836e-07)
					(1e-4,1.7369e-07)
};
				\addplot[color=darkBlue, mark=star, mark size =3.2, line width = 1]
				coordinates{
					(1e-2,1.43756e-05)
					(1e-3,1.60701e-07)
					(1e-4,1.86475e-08)
};
				\addplot[color=lightGreen, dotted, line width = 1]
				coordinates
				{(1e-1,1e-4)
					(1e-4,1e-10)};
			\end{loglogaxis}
		\end{tikzpicture}}
	\caption{\textbf{TC2 -} Convergence test for the diffusion equation on a line. Normalized $L^1$ error computed at time $t=1$ corresponding to space and time meshes with different dimensions.}
	\label{plot:line:diffusion}
	\end{figure}

	
	\subsection{TC3: Transport equation on a bifurcation}
	\label{section:R:Convergence:TransportBif}
    The first branched domain we consider is represented in Figure~\ref{figure:domain:TC3}, made of one bifurcation node and three edges of equal length $|e_k| = L = 2,\ k=1,2,3$, each further partitioned into segments of length $h$ for the space discretization. We set the advection speed $c$ directed as the respective edges and constant on each of them:
	
	\begin{equation*}
		c_k = c\big|_{e_k} =
		\begin{cases}
			10, & k=2, \\
			5, & k=1,3.
		\end{cases}
	\end{equation*}

	\noindent If we set as initial condition $u_0 = 0$ on the whole domain $\Lambda$ and inflow boundary condition $u_B=1$ on the vertex $v_2$, we are able to compute the exact solution at each time step using the method of characteristics. On $e_2$ we have the propagation of a rectangular wave with amplitude $1$ and speed $c_2$, that reaches the bifurcation point for $t = t^* = \dfrac{|e_2|}{c_2} = \dfrac{2}{10} = 0.2$:
	
	\begin{equation*}
		u_\text{ex}(x,t) =
		\begin{cases}
			1, & \text{if}\ x\leq c_2 t \ \text{and} \ 0\leq t \leq t^*, \ \text{or}\ t>t^*, \\
			0, & \text{if}\ x>c_2 t, \ 0\leq t \leq t^*.
		\end{cases}
	\end{equation*}

	\noindent Let $u^*$ be the inflow condition on the edges $e_1$ and $e_3$ through the node $I$. Then, we obtain the propagation of a rectangular wave of amplitude $u^* = \dfrac{c_2}{c_1+c_3}u_2$ and speed $c_k$ on each $e_k, \ k=1,3$, starting at time $t=t^*$. Finally, the exact solution on $e_k, \ k=1,3$, is the following:
	
	\begin{equation*}
		u_{\text{ex}}(x,t) =
		\begin{cases}
			1, & \text{if}\ x\leq c_k t, \ t>t^*, \\
			0, & \text{otherwise}.
		\end{cases}
	\end{equation*}

    \noindent Thus, we expect a piecewise constant solution on the edges of the extended graph. However, in Figure~\ref{figure:branch:transport}, representing the numerical solution at three different times for $h=10^{-2}$, we can observe that the jump discontinuities are not exactly captured. For instance, the solution at time $t=0.2$ should be $u=2$ on $e_2$ and $u=0$ on $e_1$ and $e_3$, but a little dissipation is shown in proximity of the bifurcation, due to the diffusivity of the upwind scheme.\\
    Since on each branch the numerical method is a Finite Volume scheme with upwind flux, we expect to observe similar convergence properties as in the analogous test case~\eqref{eq:tc1b} presented in Section~\ref{section:R:ConvergenceTransport} with discontinuous solution. Indeed, Figure~\ref{plot:branch:transport} shows convergence with order $\frac{1}{2}$ both in time and space of the normalized $L^1$ error.

		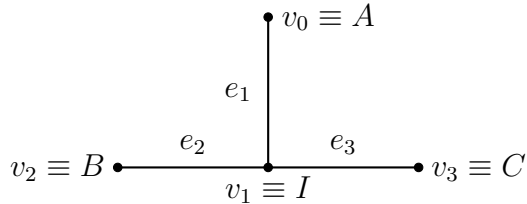
\begin{figure}[h]
		\centering
		\begin{tikzpicture}[scale=.8]
			\filldraw (0,0) circle (2pt) node [yshift=-.3cm] {$v_1\equiv I$};
			\filldraw (0,2.5) circle (2pt) node [xshift = .8cm] {$v_0\equiv A$};
			\filldraw (-2.5,0) circle (2pt) node[xshift = -.8cm] {$v_2\equiv B$};
			\filldraw (2.5,0) circle (2pt) node [xshift = .8cm] {$v_3\equiv C$};
			\draw [thick] (0,0) -- (-2.5,0) node [xshift = 1cm, yshift = .3cm] {$e_2$};
			\draw [thick] (0,0) -- (2.5,0) node [xshift = -1cm, yshift = .3cm] {$e_3$};
			\draw [thick] (0,2.5) -- (0,0) node [xshift = -.4cm, yshift = 1cm] {$e_1$};
		\end{tikzpicture}
		\caption{Simple one-dimensional domain with three edges and one bifurcation node $v_1$. The set of edges is $\mathcal{E}=\{e_1=IA,e_2=BI,e_3=IC\}$. The only source node is $B$ and the set of end nodes is $\mathcal{N}_e=\{A,C\}$. There is one bifurcation node $I$.}
		\label{figure:domain:TC3}
	\end{figure}
	
	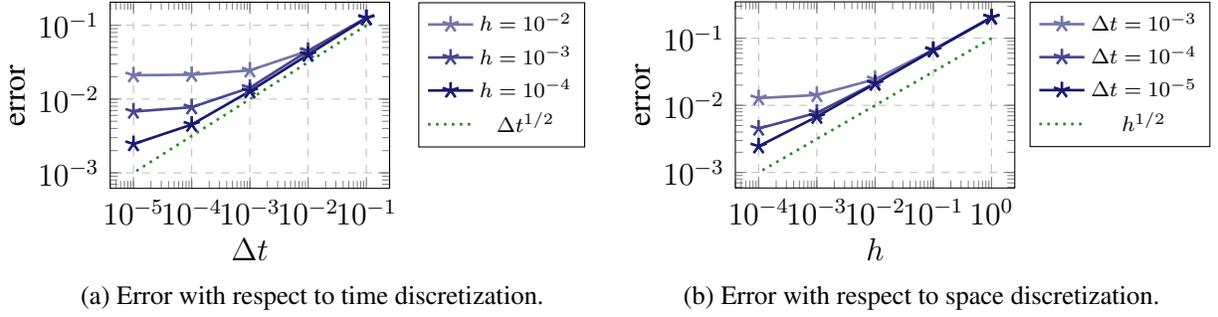
\begin{figure}
		\centering
		\subfloat[\centering Error with respect to time discretization.]
		{\begin{tikzpicture}
			\begin{loglogaxis}
				[scale = .5,
				width = .45\textwidth,
				height = .3\textwidth,
				xlabel={$\Delta t$},
				ylabel={error},
				scale only axis,
				xtick={},
				ytick={},
				ymajorgrids=true,
				xmajorgrids=true,
				grid style=dashed,
				legend style = {draw=black, font=\scriptsize, at={(1.7,1)}},
				legend entries ={
					$h = 10^{-2}$, $h =10^{-3}$, $h =10^{-4}$, 
					$\Delta t^{1/2}$},
				]
				\addplot[color=darkBlue!60, mark=star, mark size =3.2, line width = 1]
				coordinates
				{(1e-1,0.126829)
					(1e-2,0.0450554)
					(1e-3,0.0245129)
					(1e-4,0.0213946)
					(1e-5,0.021063)};
				\addplot[color=darkBlue!80, mark=star, mark size =3.2, line width = 1]
				coordinates
				{(1e-1,0.12522)
					(1e-2,0.0450554)
					(1e-3,0.0142596)
					(1e-4,0.00775385)
					(1e-5,0.00677242)};
				\addplot[color=darkBlue!, mark=star, mark size =3.2, line width = 1]
				coordinates
				{(1e-1,0.125067)
					(1e-2,0.0400273)
					(1e-3,0.0127944)
					(1e-4,0.0045098)
					(1e-5,0.00245704)};
				\addplot[color=lightGreen, dotted, line width = 1]
				coordinates
				{(1e-1,1e-1)
					(1e-5,1e-3)};
			\end{loglogaxis}
		\end{tikzpicture}
		\quad}
		\subfloat[\centering Error with respect to space discretization.]
		{\begin{tikzpicture}
			\begin{loglogaxis}
				[scale = .5,
				width = .45\textwidth,
				height = .3\textwidth,
				xlabel={$h$},
				ylabel={error},
				scale only axis,
				xtick={},
				ytick={},
				ymajorgrids=true,
				xmajorgrids=true,
				grid style=dashed,
				legend style = {draw=black, font=\scriptsize, at={(1.7,1)}},
				legend entries ={
					$\Delta t = 10^{-3}$, $\Delta t = 10^{-4}$, $\Delta t = 10^{-5}$,
					$h^{1/2}$},
				]
				\addplot[color=darkBlue!60, mark=star, mark size =3.2, line width = 1]
				coordinates
				{(1e0,0.204478)
					(1e-1,0.0674696)
					(1e-2,0.0245129)
					(1e-3,0.0142596)
					(1e-4,0.0127944)};
				\addplot[color=darkBlue!80, mark=star, mark size =3.2, line width = 1]
				coordinates
				{(1e0,0.204143)
					(1e-1,0.0664047)
					(1e-2,0.0213946)
					(1e-3,0.00775385)
					(1e-4,0.0045098)};
				\addplot[color=darkBlue, mark=star, mark size =3.2, line width = 1]
				coordinates
				{(1e0,0.204114)
					(1e-1,0.066302)
					(1e-2,0.021063)
					(1e-3,0.00677242)
					(1e-4,0.00245704)};
				\addplot[color=lightGreen, dotted, line width = 1]
				coordinates
				{(1e-0,1e-1)
					(1e-4,1e-3)};
			\end{loglogaxis}
		\end{tikzpicture}
		}
		\caption{\textbf{TC3 -} Convergence test for transport equation on a graph with one bifurcation node. Normalized $L^1$ error computed at time $t=1$ corresponding to space and time meshes with different dimensions.}
		\label{plot:branch:transport}
	\end{figure}

    \begin{figure}
        \centering
        \subfloat[\centering $t=0$]
        {\includegraphics[height=0.14\linewidth]{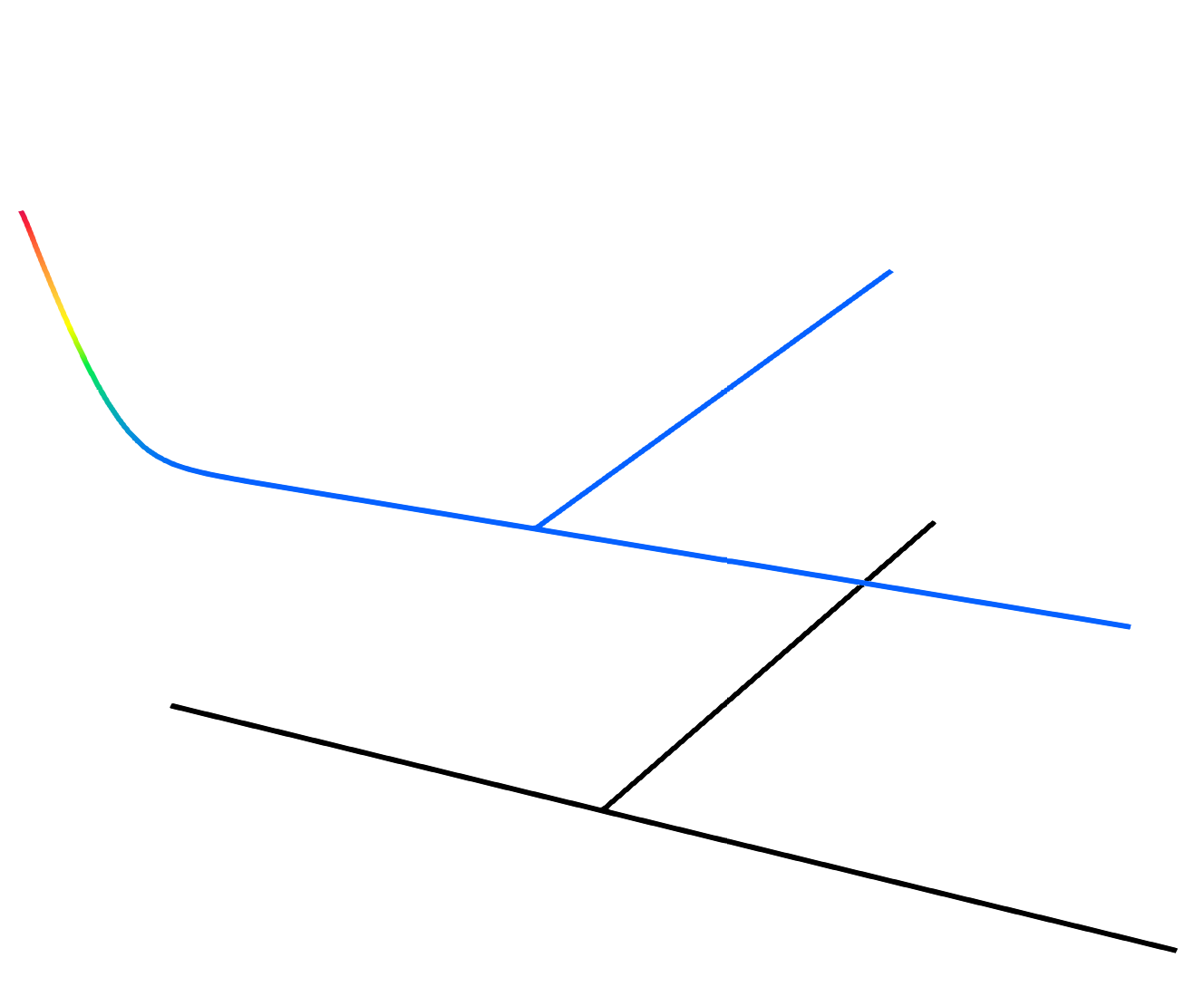}\hspace{2cm}}
        \subfloat[\centering $t=0.2$]
        {\includegraphics[height=0.14\linewidth]{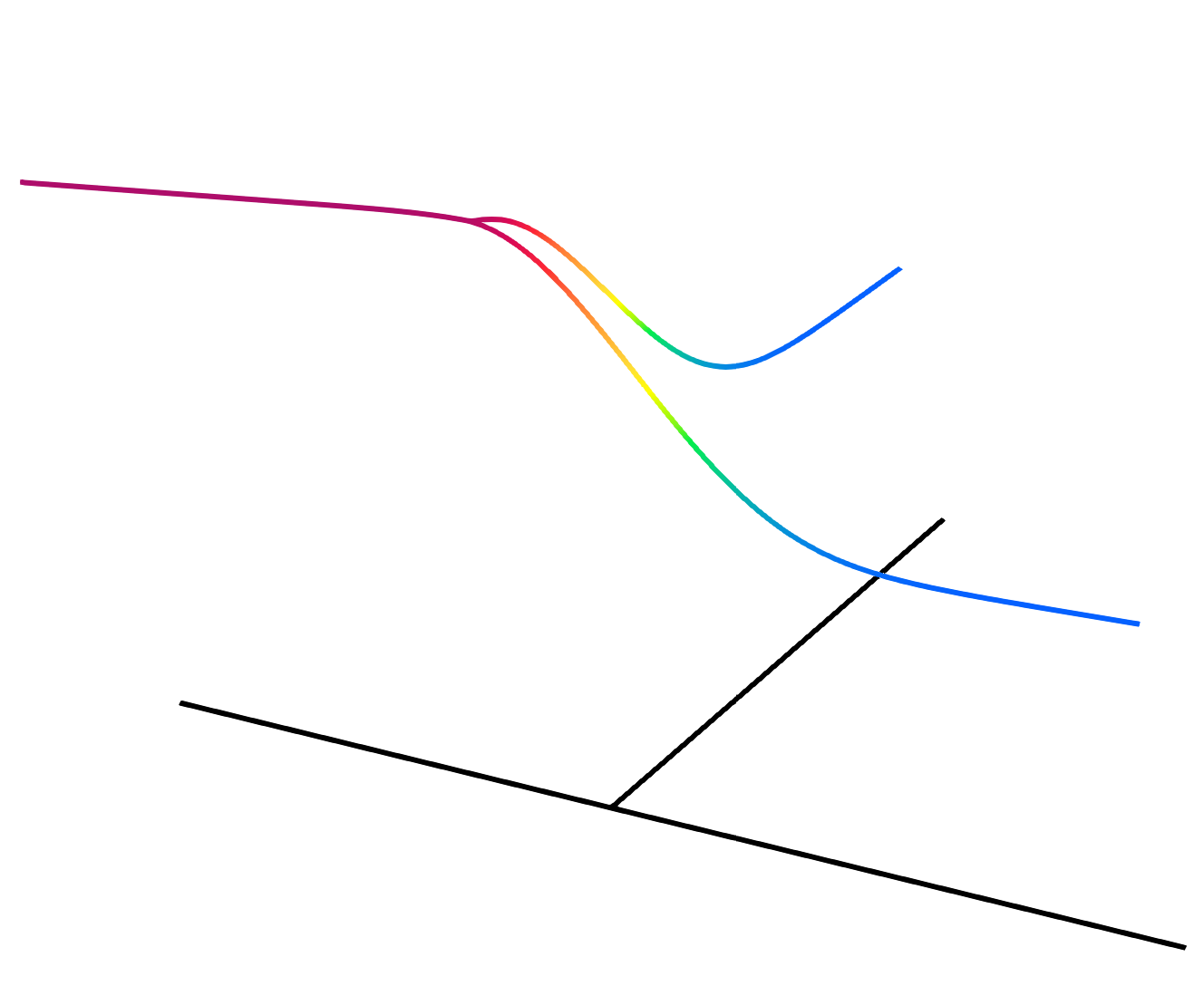}\hspace{2cm}}
        \subfloat[\centering $t=1$]
        {\includegraphics[height=0.14\linewidth]{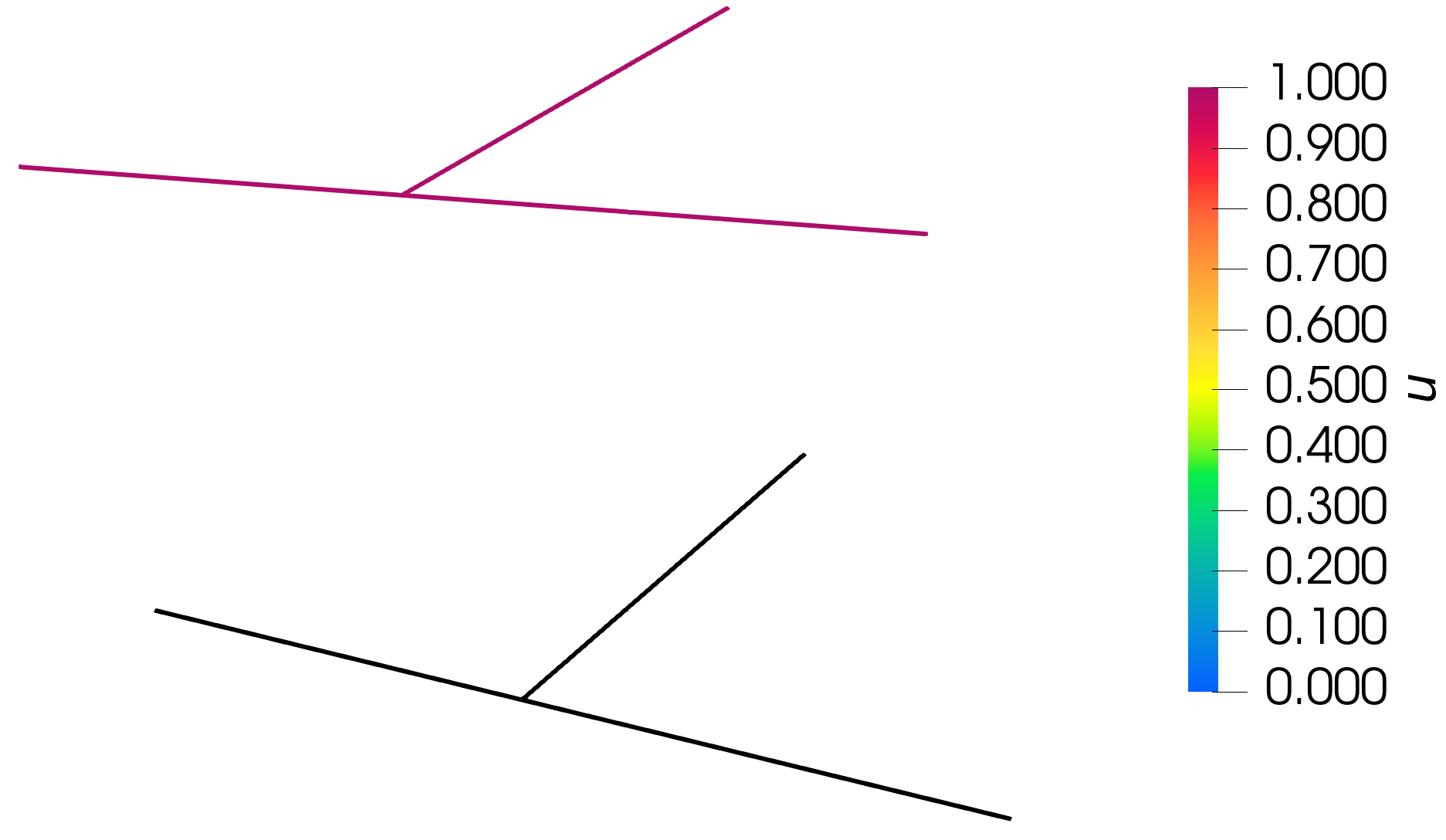}}
        \caption{\textbf{TC3 -} Numerical solution at initial time $t=0$, at $t=0.2$ when the edge $e_2$ is full and at the final time $t=1$. In black a representation of the domain $\Lambda$.}
        \label{figure:branch:transport}
    \end{figure}
	
	\subsection{TC4: Diffusion equation on a bifurcation}
	\label{section:R:Convergence:DiffusionBif}
	
	The second test case on a graph is the solution of diffusion equation~\eqref{eq:diffusion} on the domain represented in Figure~\ref{figure:domain:TC4} with four edges and one bifurcation node.\\
	We set diffusion coefficient $\nu=4$, impose homogeneous Dirichlet boundary conditions on the set of boundary nodes $\mathcal{N}_b = \{v_0,v_2,v_3,v_4\}$ and initial condition $u^0(s) = \cos\left(\frac{\pi}{2}s\right)$. This problem was constructed by starting from the following exact solution:
	
	\begin{equation*}
		u_\text{ex}(s,t) = \cos\left(\frac{\pi}{2}s\right)e^{-\pi^2 t}.
	\end{equation*}

	\noindent Due to the symmetry of the domain and of the solution, we expect the numerical solver to behave as on two separate straight lines and have similar performances as in the test case~\eqref{eq:tc2} presented in Section~\ref{section:R:ConvergenceDiffusion}. Indeed, in Figure~\ref{plot:branch:diffusion} we can see that the orders of convergence 1 and 2 for time and space, respectively, are shown also in presence of a bifurcation. Finally, the numerical solution at three different times is represented in Figure~\ref{figure:branch} and shows a good approximation of the initial datum and of the expected decrease in the peak of the solution.
	
	\begin{figure}[h]
		\centering
		\begin{tikzpicture}[scale=.7]
			\filldraw (0,0) circle (2pt) node [yshift=-.3cm,xshift=.8cm] {$v_1\equiv I$};
			\filldraw (0,2.5) circle (2pt) node [xshift = .8cm] {$v_0\equiv A$};
			\filldraw (-2.5,0) circle (2pt) node[xshift = -.8cm] {$v_2\equiv B$};
			\filldraw (2.5,0) circle (2pt) node [xshift = .8cm] {$v_3\equiv C$};
			\filldraw (0,-2.5) circle (2pt) node [xshift = .8cm] {$v_3\equiv D$};
			\draw [thick] (0,0) -- (-2.5,0) node [xshift = 1cm, yshift = .3cm] {$e_2$};
			\draw [thick] (0,0) -- (2.5,0) node [xshift = -1cm, yshift = .3cm] {$e_3$};
			\draw [thick] (0,2.5) -- (0,0) node [xshift = -.4cm, yshift = 1cm] {$e_1$};
			\draw [thick] (0,-2.5) -- (0,0) node [xshift = -.4cm, yshift = -1cm] {$e_4$};
		\end{tikzpicture}
		\caption{Simple one-dimensional domain with four edges and one bifurcation node $v_1$. The set of edges is $\mathcal{E}=\{e_1=IA,e_2=BI,e_3=IC,e_4=ID\}$. All the boundary nodes $A,\ B,\ C,\ D$ are source nodes and the only bifurcation node is $I$.}
		\label{figure:domain:TC4}
	\end{figure}
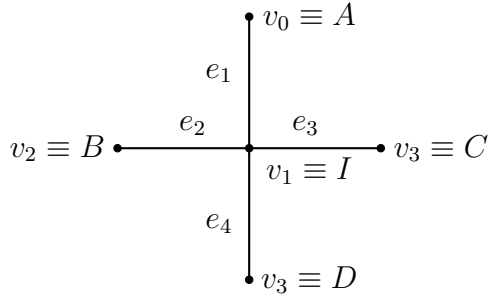
	
		\begin{figure}
		\centering
		\subfloat[\centering Error with respect to time discretization.]
		{\begin{tikzpicture}
			\begin{loglogaxis}
				[scale = .5,
				width = .45\textwidth,
				height = .3\textwidth,
				xlabel={$\Delta t$},
				ylabel={error},
				scale only axis,
				xtick={},
				ytick={},
				ymajorgrids=true,
				xmajorgrids=true,
				grid style=dashed,
				legend style = {draw=black, font=\scriptsize, at={(1.7,1)}},
				legend entries ={
					$h =10^{-1}$, $h = 10^{-2}$, $h =10^{-3}$, $\Delta t$},
				]
				\addplot[color=darkBlue!60, mark=star, mark size = 3.2,
				line width = 1]
				coordinates
				{(1e-1,0.00253062)
					(1e-2,7.73285e-5)
					(1e-3,7.22771e-6)
					(1e-4,1.31537e-6)
					(1e-5,7.47465e-7)
					(1e-6,6.76401e-7)
				};	
				\addplot[color=darkBlue!80, mark=star, mark size =3.2, line width = 1]
				coordinates
				{(1e-1,0.00252326)
					(1e-2,7.62556e-05)
					(1e-3,6.53671e-06)
					(1e-4,6.49348e-07)
					(1e-5,8.38528e-08)
					(1e-6,1.30978e-08)
				};	
				\addplot[color=darkBlue, mark=star, mark size =3.2, line width = 1]
				coordinates
				{(1e-1,0.00252319)
					(1e-2,7.63459e-05)
					(1e-3,6.52985e-06)
					(1e-4,6.42705e-07)
					(1e-5,7.72338e-08)
					(1e-6,6.48144e-09)
				};	
				\addplot[color=darkRed, dotted, line width = 1]
				coordinates	
				{(1e-1,1e-4)
					(1e-6,1e-9)};	
			\end{loglogaxis}
		\end{tikzpicture}\quad}
		\subfloat[\centering Error with respect to space discretization.]
		{\begin{tikzpicture}
			\begin{loglogaxis}
				[scale = .5,
				width = .45\textwidth,
				height = .3\textwidth,
				xlabel={$h$},
				ylabel={error},
				scale only axis,
				xtick={},
				ytick={},
				ymajorgrids=true,
				xmajorgrids=true,
				grid style=dashed,
				legend style = {draw=black, font=\scriptsize, at={(1.7,1)}},
				legend entries ={
                    $\Delta t = 10^{-4}$, $\Delta t = 10^{-5}$, $\Delta t = 10^{-6}$, 
					$h^2$},
				]
				\addplot[color=darkBlue!60, mark=star, mark size =3.2, line width = 1]
				coordinates
				{(0.25e0,8.71363e-05)
					(0.25e-1,7.47465e-7)
					(0.25e-2,8.38528e-08)
					(0.25e-3,7.72338e-08)
				};
				\addplot[color=darkBlue!80, mark=star, mark size =3.2, line width = 1]
				coordinates
				{(0.25e0,8.70276e-05)
					(0.25e-1,6.76401e-7)
					(0.25e-2,1.30978e-08)
					(0.25e-3,6.48144e-09)
				};
				\addplot[color=darkBlue, mark = star, mark size = 3.2, line width = 1]
				coordinates
				{(0.25e0,8.70189e-05)
					(0.25e-1,6.70603e-7)
					(0.25e-2,7.32389e-09)};
				\addplot[color=lightGreen,dotted, line width = 1]
				coordinates
				{(1e0,1e-4)
					(1e-4,1e-12)};
			\end{loglogaxis}
		\end{tikzpicture}}
		\caption{\textbf{TC4 -} Convergence test for diffusion equation on a graph with one bifurcation node. $L^1$ error computed at time $t=1$ corresponding to space and time meshes with different dimensions.}
		\label{plot:branch:diffusion}
	\end{figure}
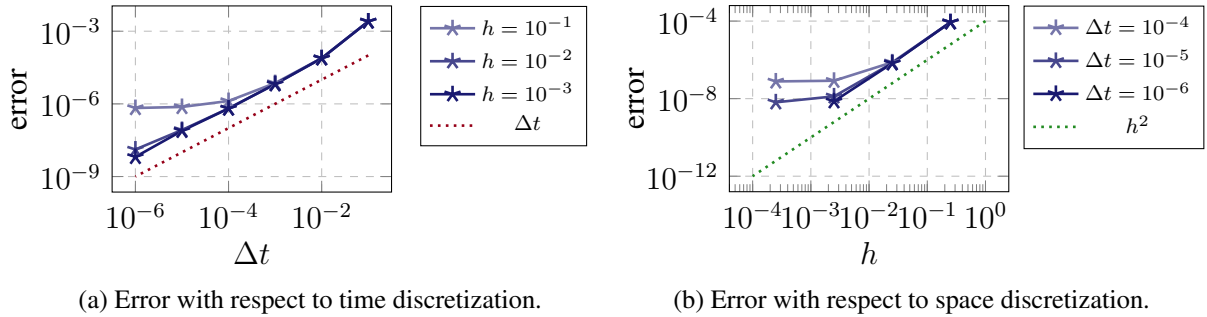

    \begin{figure}
        \centering
        \subfloat[\centering $t=0$]
        {\includegraphics[height=0.2\linewidth]{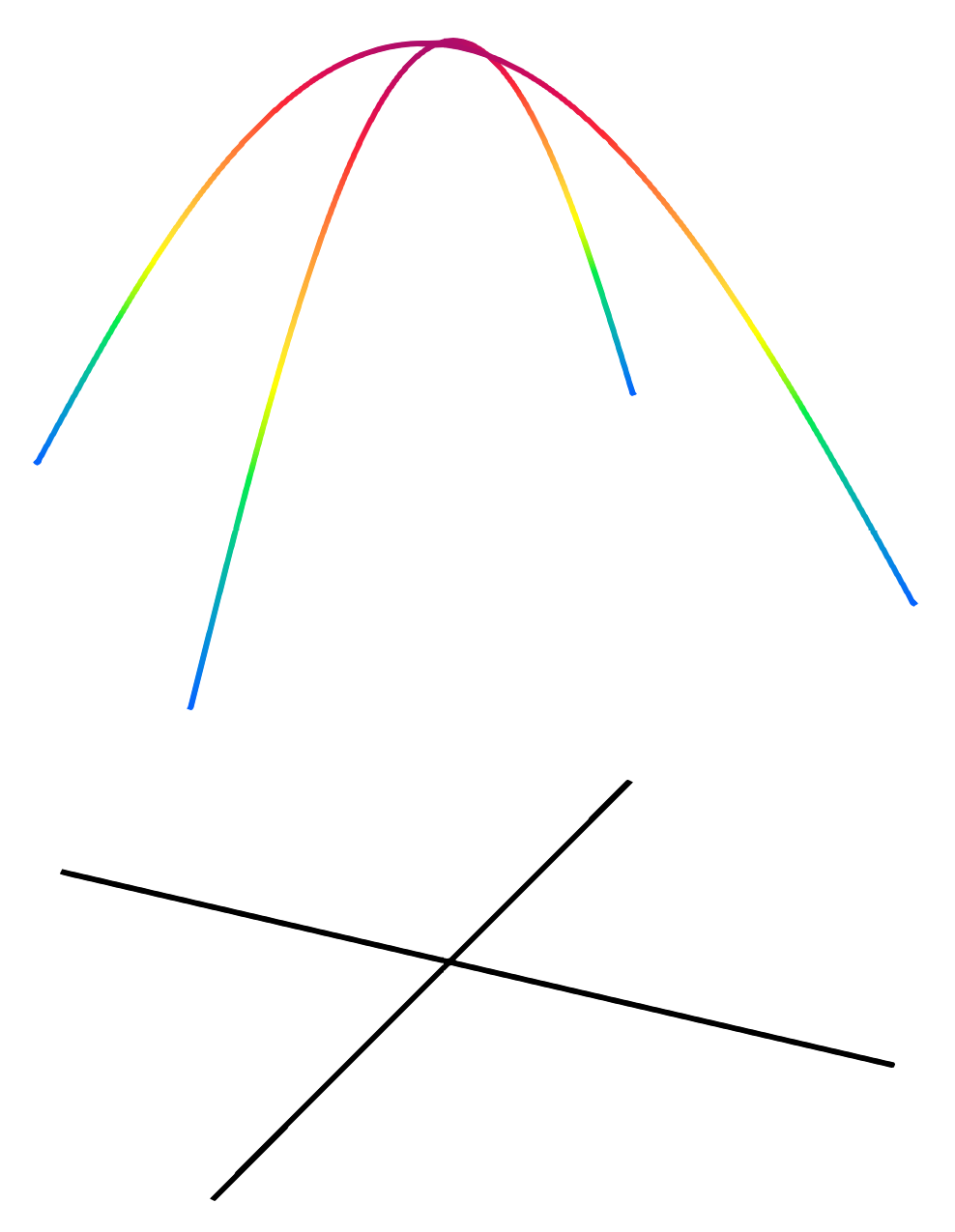}\hspace{2cm}}
        \subfloat[\centering $t=0.2$]
        {\includegraphics[height=0.2\linewidth]{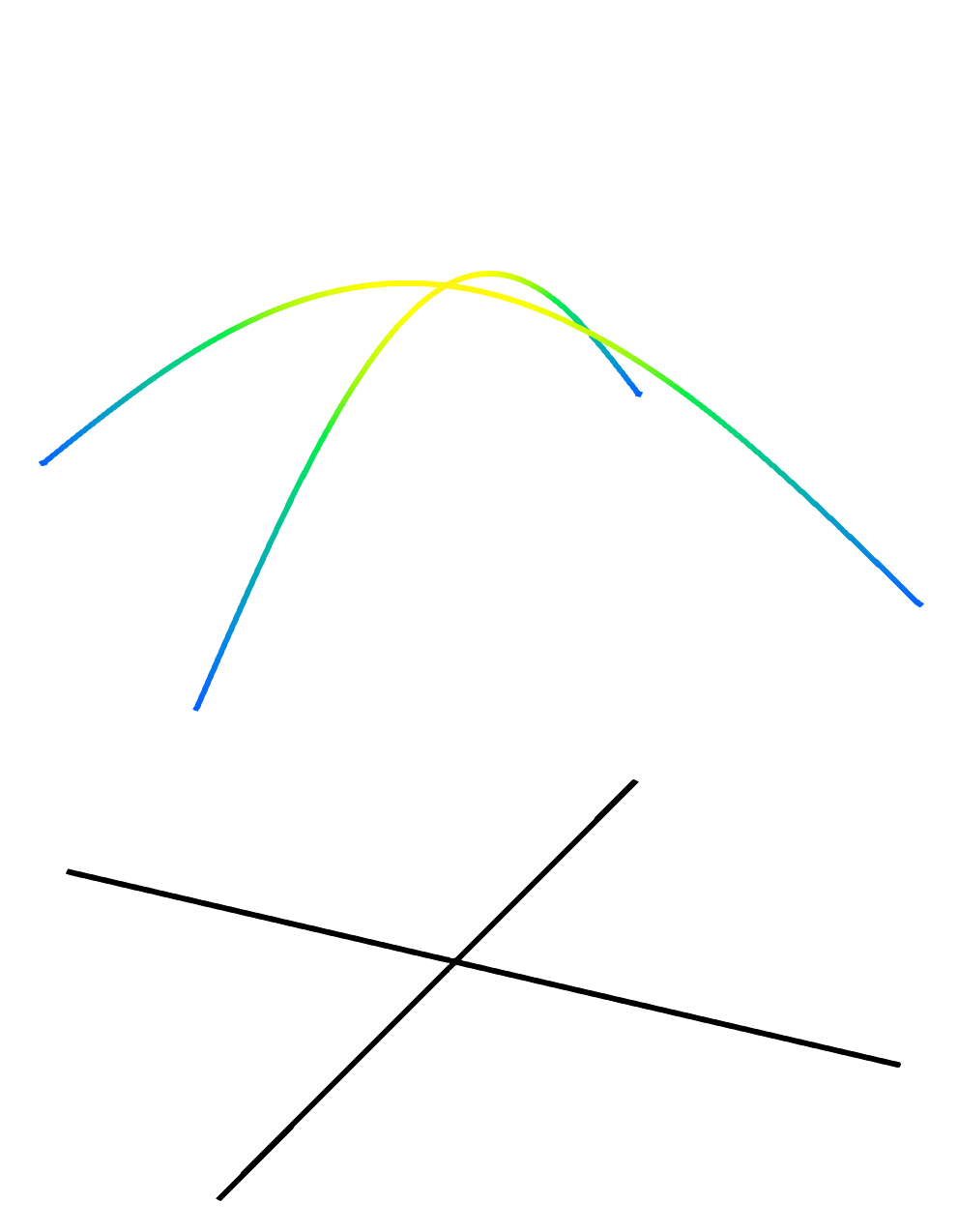}\hspace{2cm}}
        \subfloat[\centering $t=1$]
        {\includegraphics[height=0.2\linewidth]{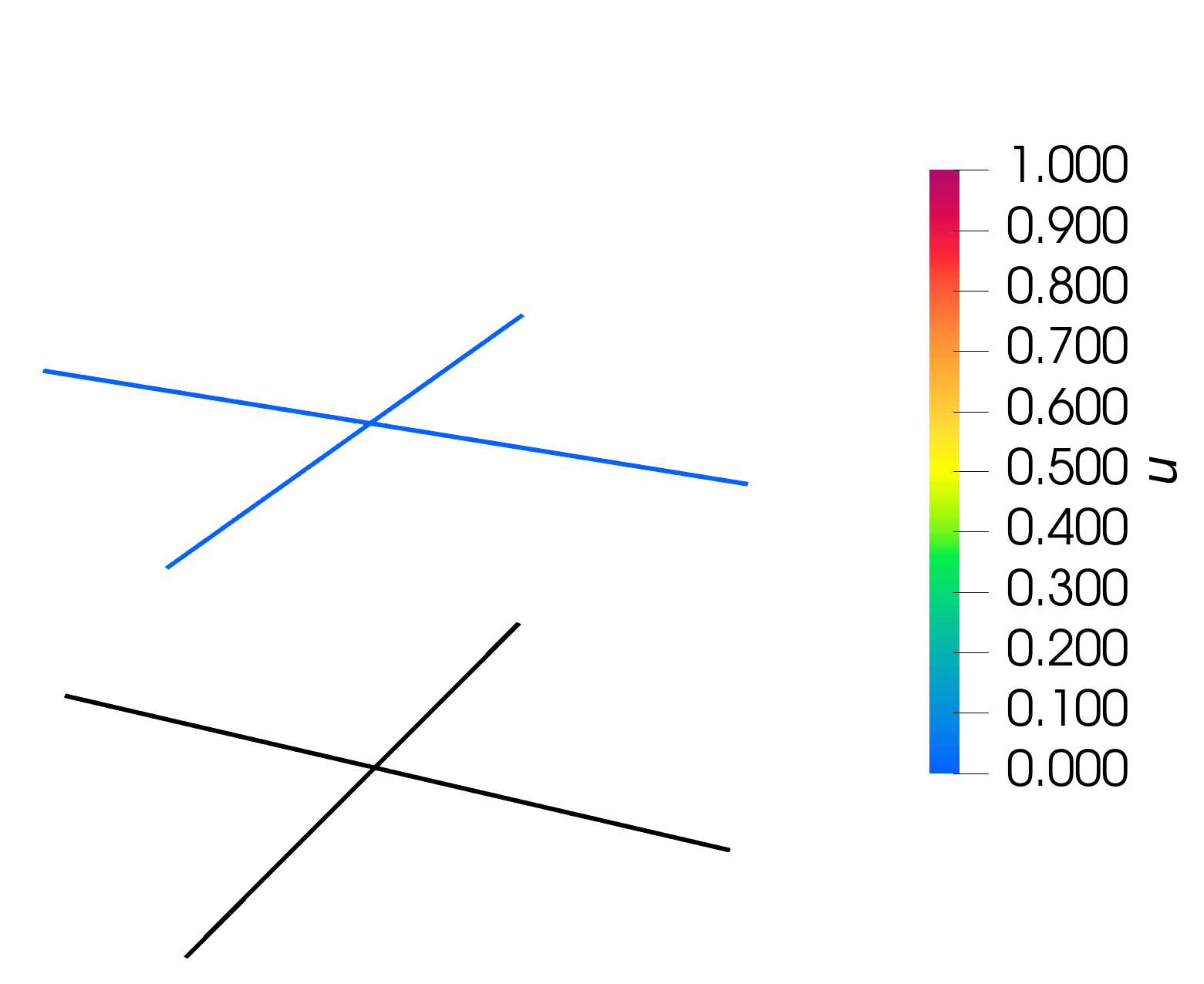}}
        \caption{\textbf{TC4 -} Numerical solution at initial time $t=0$, at $t=0.1$ and at the final time $t=1$. In black a representation of the domain $\Lambda$.}
        \label{figure:branch:diffusion}
    \end{figure}
 
	
	\subsection{TC5: Drift-diffusion equation on Electrical Treeing}
	\label{section:R:TreeingDiffusionTransport}
    \begin{figure}
        \centering
        \includegraphics[height=0.18\linewidth]{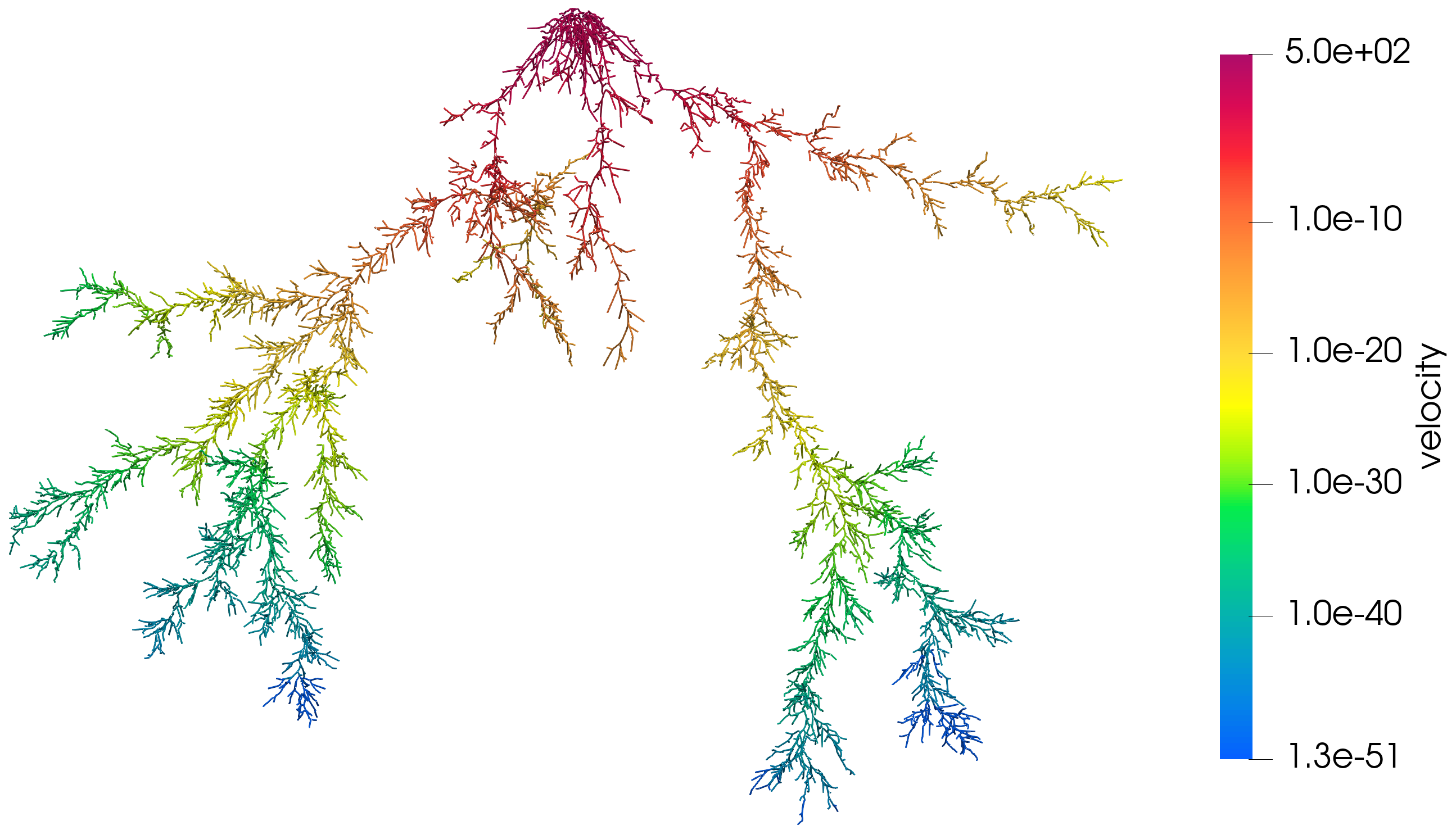}
        \caption{\textbf{TC5 -} Drift velocities on the graph edges. The maximum value is on the edges connected to the inflow node, and it decreases going towards the outflow note, until the minimum value of order $10^{-51}$.}
        \label{figure:TC5:velocity}
    \end{figure}
 
    \begin{figure}
        \centering
        \subfloat[\centering $t=0.1 s$]
        {\includegraphics[height=0.18\linewidth]{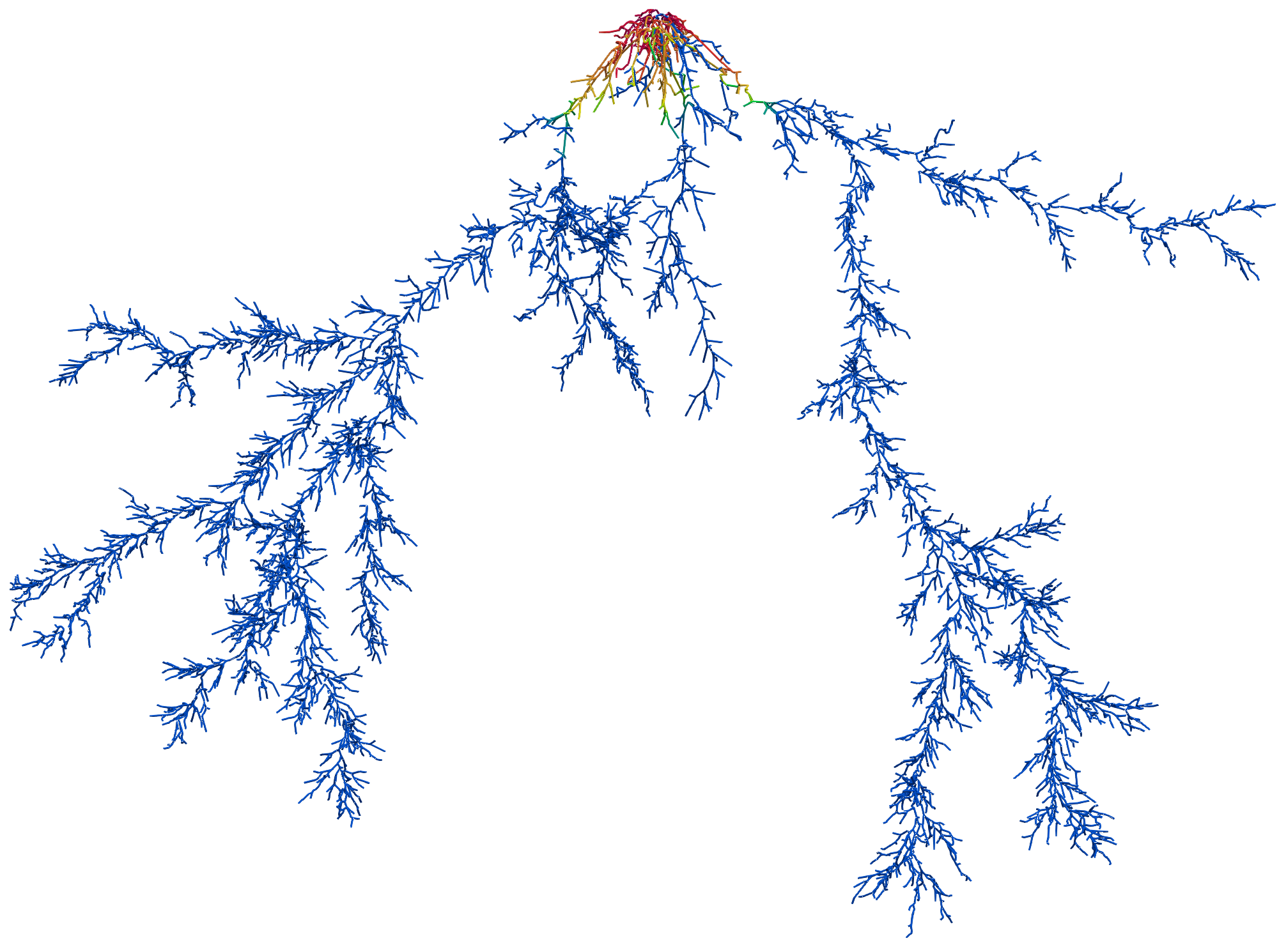}\hspace{2cm}}
        \subfloat[\centering $t=500 s$]
        {\includegraphics[height=0.18\linewidth]{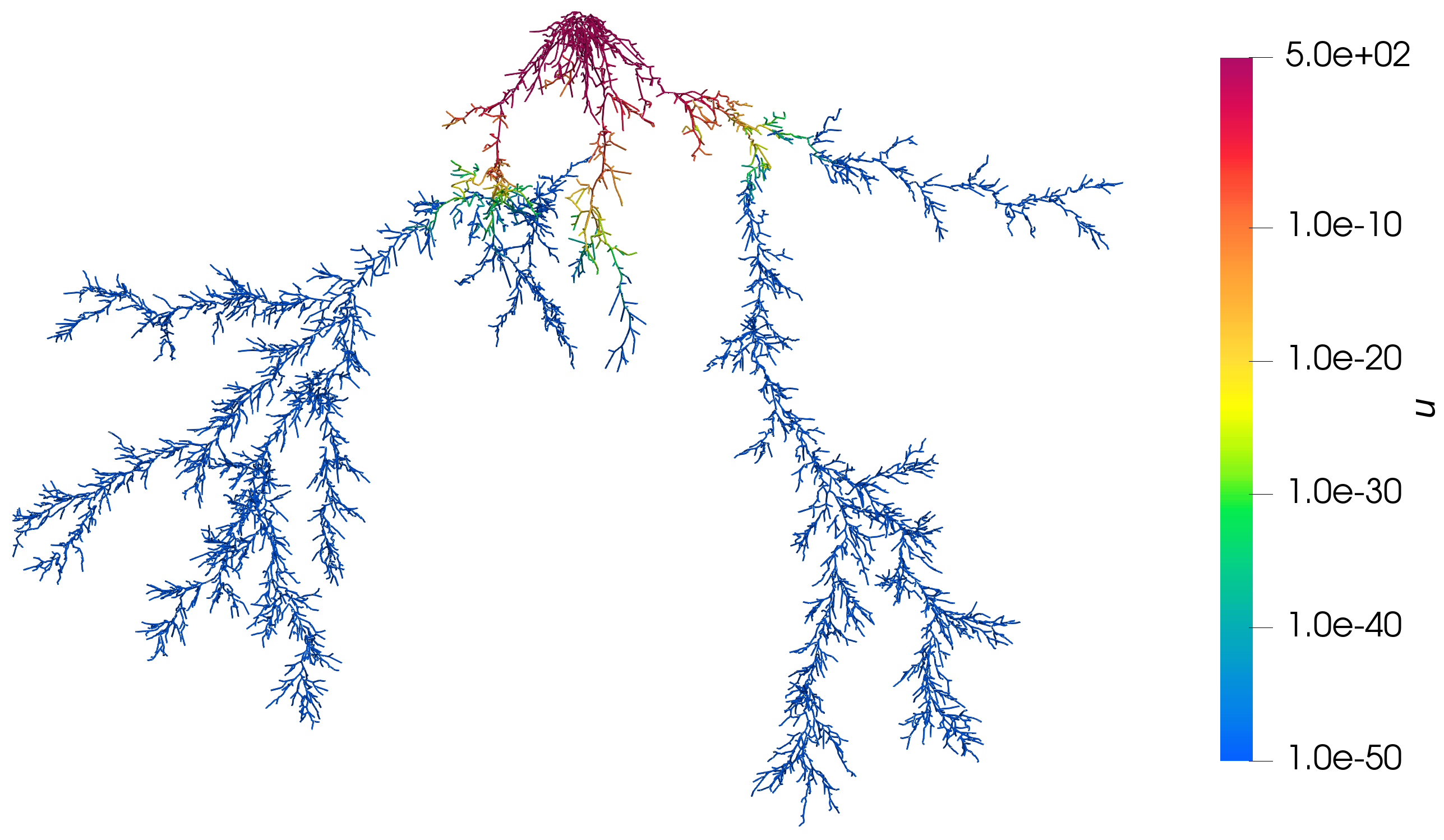}}
        \caption{\textbf{TC5 -} Numerical solution at first time step $t=0.1 s$ and at the final time $t=500 s$ on the domain of an an electrical treeing.}
        \label{figure:TC5}
    \end{figure}
    Finally, we solve the drift-diffusion problem~\eqref{eq:diffusion_transport} on a ramified domain representing the skeleton of a typical electrical treeing, experimentally obtained by X-ray computed tomography~\cite{schurch2014imaging} on an existing defect in an electrical cable, resulting in a graph composed of 12544 edges. Such an extended defect can hardly be discretized in 3D, due to it geometrical complexity. Thus, we need an ad hoc numerical solver for the reduction to PDEs on a 1D graph of problems defined on a 3D electrical treeing.\\
    This applied test case is derived by a geometrical reduction to one dimension of the model describing the movement of charge densities across the 3D electrical treeing, where the unknown $u$ represents the integral mean of the charge concentration on transversal sections of the 3D domain. We consider a constant electron diffusion coefficient $\nu = 0.5 \times 10^{-6} \frac{m^2}{\mu s}$, while the value of the transport speed on each edge $e_k$ of the graph is given by $c_k = \mu |\mathbf{E}_k|$, $k=1,\dots,\mathcal{N}_e$, where $\mu = 10^{-5} \frac{m^2}{kV\ \mu s}$ denotes the electronic mobility of the material composing the domain and $|\mathbf{E}_k|$ the magnitude of the electric field along $e_k$, $k=1,\dots,\mathcal{N}_e$. We assume that the electric field is different on each edge of the graph, respecting the compatibility condition~\eqref{eq:th:tr:c}. The electric field on the edges connected to the inflow node is equal to $50\ kV$, and it is split equally among its outgoing neighbors, at each bifurcation. Figure~\ref{figure:TC5:velocity} shows the value of the velocity corresponding to each edge of the graph.\\
    We assume that the set of source nodes consists of a single node, given by the root of the tree, where we impose the inflow electron concentration through the Dirichlet condition $u = 100\frac{1}{m^2}$, and simulate over the time interval $(0, 500 s]$, starting from the initial condition $u_0 = 0\frac{1}{m^2}$. The choice of such a long time interval for the simulation is for illustrative purposes only, due to the low transport velocity as we approach the periphery of the tree.\\
    For the time discretization we consider a time step $dt = 0.1 s$, while the edges of the space mesh coincide with the edges of the graph, not further partitioned. The final result is shown in Figure~\ref{figure:TC5}, where the concentration imposed as inflow condition on top of the domain is spread across the whole graph. Not all the edges are filled because, as displayed in Figure~\ref{figure:TC5:velocity}, the transport velocity becomes close to zero at the middle of the graph. However, we expect that, in a long time, the graph will all be filled, i.e. the concentration on all the edges will eventually be the same as prescribed by the inflow condition.\\
    We were able to simulate this problem in approximately 2 hours in parallel on 6 cores on a laptop with 16 GiB RAM, 11th Gen Intel(R) Core(TM) i7-11. In this test case we were able to deal with a very complex and extended graph, solving the problem on a relatively coarse time grid and without further partitioning the edges of the graph, obtaining a positive solution at each time step, in reasonable time. 
	
	\section{Conclusions and Discussion}
	\label{section:Conclusions}

    We have defined two numerical schemes based on Finite Volumes for the approximate solution of a linear transport equation and a drift-diffusion problem on one-dimensional graph domains, with implicit time discretization. We have proven the existence and uniqueness of the solution to both discrete problems and that, starting from a non-negative initial condition and source term, the solutions at each time step remain non-negative. Moreover, we have shown that the numerical fluxes are consistent at the graph nodes. Positivity of the approximate solution at each time step is necessary for the problems of our interest, concerning plasma, where the presented methods are coupled with chemical solvers, that cannot deal with negative concentrations.\\
    The two methods are extensions of a Finite Volume upwind scheme for the transport equation and finite differences for the diffusion part. We have observed that the convergence results of the proposed methods on a one-dimensional domain consisting in a straight line coincide with the theoretical results for these two schemes and that this property is also extended to graph domains with bifurcation nodes.\\
    Finally, we have applied the two numerical methods to the solution of a transport equation and a drift-diffusion problem on the geometry of an electrical tree, and obtained a solution in a very short computational time, which was our starting motivation. The methods present good performance on real complex geometries and moderate computational cost, despite the high number of branches. Moreover, we have made a parallel implementation of each solver, allowing to exploit resources of high performance computers. Unfortunately, both the upwind and implicit Euler methods are diffusive, but our choice fell on these schemes because of the necessity of a moderate computational cost. In fact, the possibility of adopting large time spacings $\Delta t$, due to the stability of the time numerical scheme, helps in the reduction of the computational time.\\
    We aim at incorporating the presented problems, describing the evolution of charge densities in non-thermal plasma, into a more comprehensive framework modelling the whole phenomenon of Partial Discharges and the evolution of the electrical treeing. The goal is to overcome the employment of semi-empirical schemes, used so far in th literature on the topic, keeping a controlled computational cost~\cite{villa2021uncoupled, villa2022simulation}. However, a possible extension can be the application of higher order methods, such as Discontinuous Galerkin in space and Crank-Nicholson, which is still unconditionally stable, in time.

    \section{Acknowledgments}
    \label{section:acknowledgments}
    
    This work has been financed by the Research Found for the Italian Electrical System under the Contract Agreement between RSE and the Ministry of Economic Development. The mesh of the treeing structure is courtesy of prof. R. Schurch, Universidad T\'ecnica Federico Santa Mar\'ia Valpara\'iso, Chile.
	
	\printbibliography[
	heading=bibintoc,
	title={References}
	]

\end{document}